\documentclass[11pt]{amsart}
\usepackage{amsmath}
\usepackage{amsaddr}
\usepackage{amsfonts}
\usepackage{amsthm}
\usepackage{amssymb}
\usepackage{mathrsfs}
\usepackage{verbatim}
\usepackage{enumerate}
\usepackage{graphicx}
\usepackage{longtable}
\usepackage[all,cmtip]{xy}
\usepackage{upref}
\usepackage{hyperref}
\usepackage[hang,small,bf]{caption}
\usepackage{marginnote}
\usepackage{a4wide}
\usepackage{mathtools}
\usepackage{xfrac}
\usepackage{bm}

\newtheorem{dfn}{Definition}[section]
\newtheorem{thm}[dfn]{Theorem}

\newtheorem{rmk}[dfn]{Remark}

\newtheorem{cor}[dfn]{Corollary}
\newtheorem{que}[dfn]{Question}
\newtheorem{con}[dfn]{Conjecture}
\newtheorem{exa}[dfn]{Example}
\newtheorem{ex}[dfn]{Exercise}

\numberwithin{equation}{section}

\newcommand{\Z}{{\mathbb Z}}

\newcommand{\R}{{\mathbb R}}
\newcommand{\N}{{\mathbb N}}

\newcommand{\C}{{\mathbb C}}

\newcommand{\vol}{\mathrm{vol\,}}

\newcommand{\dd}{{\mathrm d}}

\newcommand{\p}{\partial}

	\author{Gabriele Benedetti}\address{Mathematisches Institut, Ruprecht-Karls-Universit\"at Heidelberg}\email{gbenedetti@mathi.uni-heidelberg.de}\thanks{This work is supported by the Deutsche Forschungsgemeinschaft (DFG, German Research Foundation) under Germany's Excellence Strategy EXC 2181/1-390900948 (the Heidelberg STRUCTURES Excellence Cluster), the Collaborative Research Center SFB/TRR 191 - 281071066	(Symplectic Structures in Geometry, Algebra and Dynamics), and the Research Training Group RTG 2229 - 281869850 (Asymptotic Invariants and Limits of Groups and Spaces).}

\title[Systolic inequalities: From Riemannian metrics to Symplectic manifolds]{First steps into the world of systolic inequalities:\\ \smallskip From Riemannian to Symplectic geometry}
\begin{document}
\begin{abstract}
Our aim is to give a friendly introduction to systolic inequalities. We will stress the relationships between the classical formulation for Riemannian metrics and more recent developments related to symplectic measurements and the Viterbo conjecture. This will give us a perfect excuse to introduce the reader to some important ideas in Riemannian and symplectic geometry. 
		\smallskip
		
		\noindent\textsc{Disclaimer.} These notes are an expanded version of two talks given at the Dutsch Differential Topology and Geometry Seminar on November 27, 2020. I am grateful to the organizers \'Alvaro del Pino G\'omez, Federica Pasquotto, Thomas Rot and Robert Vandervorst for giving me the opportunity to speak at this wonderful event. The target reader of these notes is a bachelor/master student having some basic knowledge of manifolds (curves and surfaces should mostly suffice). This is a very preliminary version and I am not an expert in the field: If you spot any mistake, notice some missing references, or have questions on the exercises please drop me a message at the email address below.
	\end{abstract}
	\maketitle
\begin{figure}[h]
	\vspace{-20pt}\includegraphics[width=.7\textwidth]{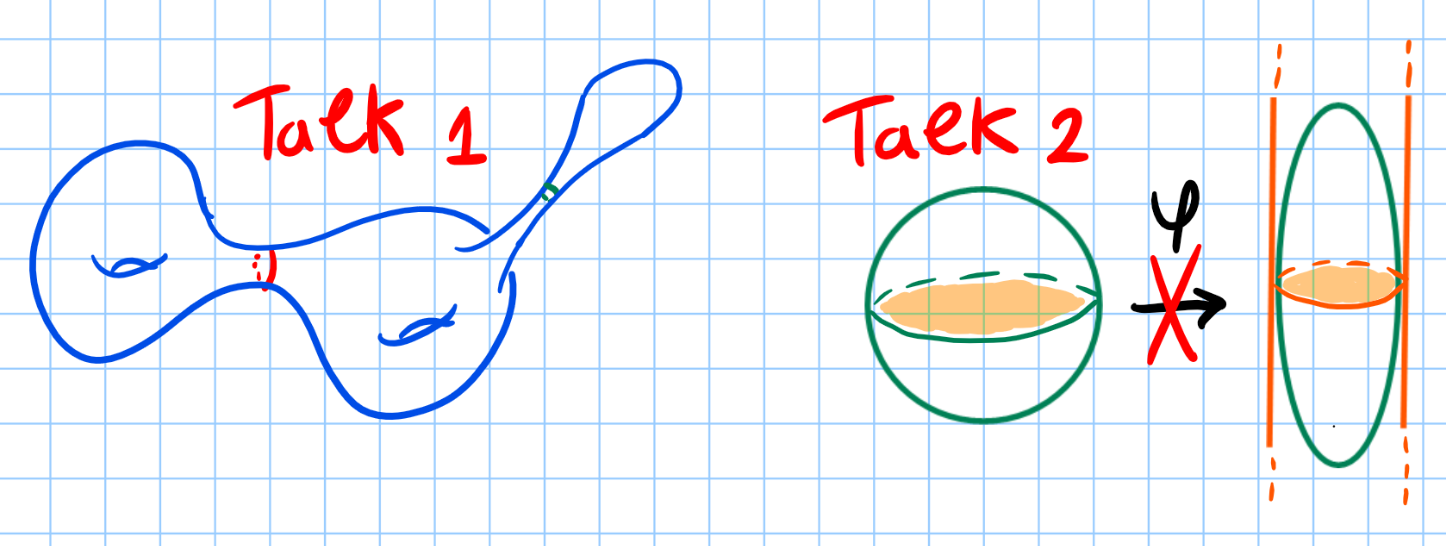}
\end{figure}
\vspace{-15pt}

\section{First talk: Systolic inequalities in Riemannian geometry}
Excellent references to systolic inequalities in Riemannian geometry are the short introductory paper \cite{Ber08} by Marcel Berger and the survey \cite{Gut10} by Larry Guth which discuss how systolic inequalities fit with other deep problems in geometry. This first talk is mainly inspired by Chapter 7.2 in \cite{Ber03}, where you will find many more beautiful results and nice pictures. A monograph dedicated to systolic inequalities is \cite{Kat07} by Mikhail Katz, which also contains interesting historical sketches. We seize the occasion to express our gratitude to Mikhail Katz for his precious comments on these notes and for highlighting several inaccuracies.
\subsection{The setting}
Riemannian geometry aims at defining quantities like angles, distances, areas, volumes, etc.\ on manifolds, namely spaces which locally look like the euclidean space $\R^m$ (for instance think at an embedded surface in three-dimensional space). So let $M$ be a manifold, which we take compact, without boundary and of dimension $m\geq 2$ (for $m=1$, $M$ would be just a circle).

To do geometry on the manifold $M$ we consider a Riemannian metric $g$ on $M$, namely a scalar product on each tangent space $T_pM$ (we will denote by $|\cdot|_g$ the corresponding norm) varying smoothly with the point $p\in M$. We compute the length of curves $\gamma:[a,b]\to M$ by 
\[
\ell_g(\gamma)=\int_a^b|\dot\gamma|_g\dd t
\]
and more generally the $n$-dimensional volume $\vol_{\!g}(N)$ of $n$-dimensional submanifolds $N\subset M$: 
\[
\vol_{\!g}(N)=\int_N\dd\vol_{\!N,g},
\]
where $\dd\vol_{\!N,g}$ is a naturally defined density (thus, for curves $\dd\vol_{\!\gamma,g}=|\dot\gamma|\dd t$).

However, there is no unique Riemannian metric on $M$ and therefore we consider the space
\[
\mathcal R(M):=\big\{\,g\text{ Riemannian metric on }M\,\big\}.
\]
This space is infinite dimensional even after identifying metrics which are related to each other by an isometry or a rescaling. In the following discussion, we will often tacitly make this identification. We will see that the systolic inequality will give us a way to study the space $\mathcal R(M)$. First, we consider some examples.

\subsection{Metrics on surfaces}

Let us start by looking at the case of surfaces $m=2$ in more detail. If $M$ is orientable, we have the sphere $S^2$, the torus $T^2$ which is obtained from the sphere attaching one handle (cylinder), and more generally an orientable surface $\Sigma_k$ of genus $k$ obtained attaching $k$ handles to the sphere. If $M$ is not orientable, we have the real projective plane $P^2$ obtained identifying pairs of antipodal points on the sphere, the Klein bottle $K^2$ obtained attaching one cross-cap (Möbius strip) to $P^2$, and more generally an unorientable surface $\tilde\Sigma_k$ obtained attaching $k$ cross-caps to the projective plane.
\begin{figure}[h]
	\includegraphics[width=.9\textwidth]{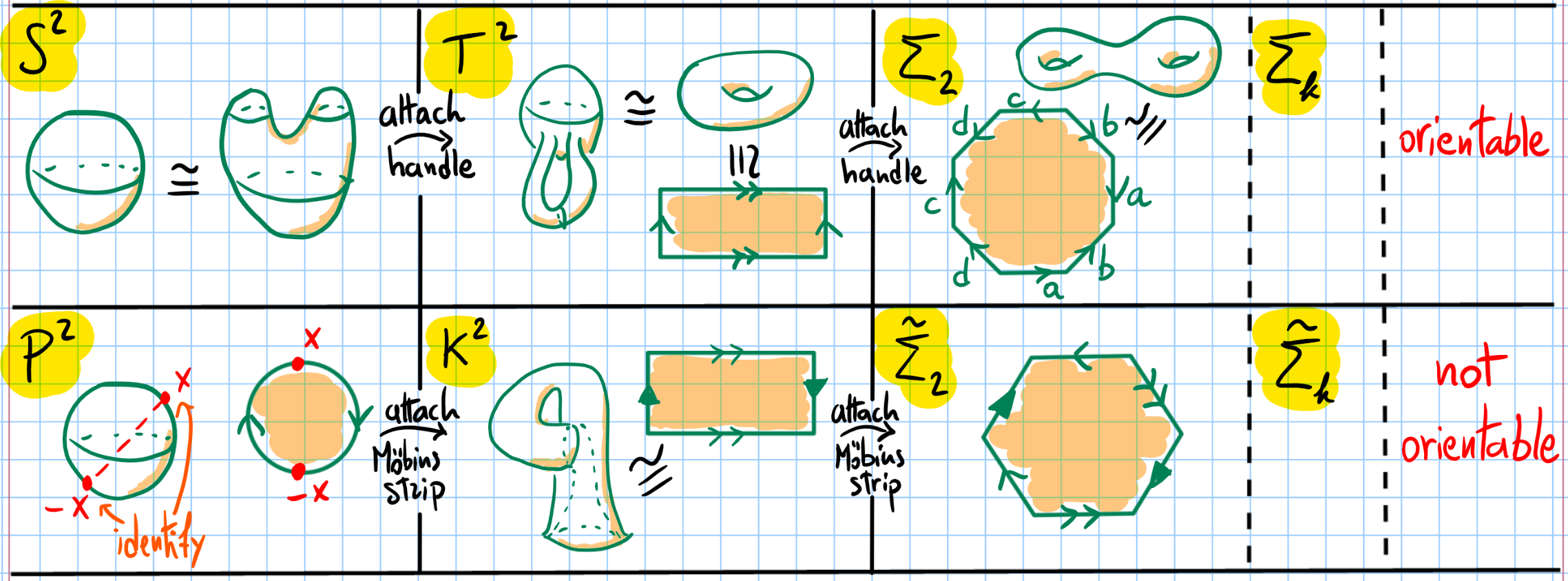}\caption{Topological classification of surfaces}
\end{figure}

For orientable surfaces $M$, one can get an idea of the richness of the space $\mathcal R(M)$ by considering the metrics induced by different embeddings of $M$ inside $\R^3$. Have a look, for instance at the embedded tori above.
\begin{figure}
\includegraphics[width=.8\textwidth]{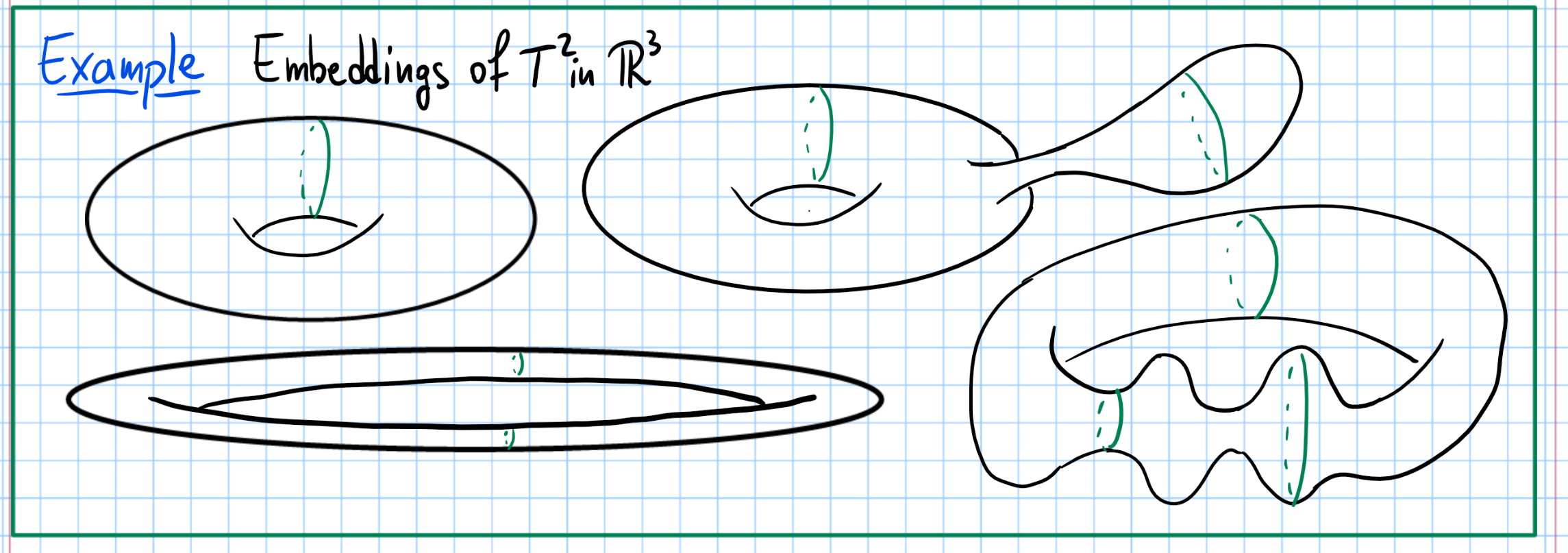}\caption{Some tori embedded in $\R^3$ (inspired by \cite{Gut10})}
\end{figure}
One could also focus on special metrics inside $\mathcal R(M)$ such as metrics of constant curvature, for which intuitively the geometry looks the same in every direction. The sign of the curvature depends only on the topology of $M$ by the Gauss--Bonnet theorem. On $S^2$ and $P^2$, the curvature is positive and we only have one such metric up to homothety (and isometry) coming from the round sphere in $\R^3$.

On $T^2$ and $K^2$ metrics of constant curvature are flat, namely they have curvature zero. On $T^2$ flat metrics cannot be obtained embedding the torus in $\R^3$ (why?). Instead, they are constructed gluing the opposite side of a parallelogram in the plane via translation. Indeed, the Euclidean metric on the parallelogram induces a Riemannian metric on the torus since translations are isometries of the Euclidean metric. Still how to classify these flat metrics up to isometry requires a bit of thought since you can cut the parallelogram along the diagonal and glue two of the parallel sides to get an isometric torus.
\begin{figure}[h]
	\includegraphics[width=.4\textwidth]{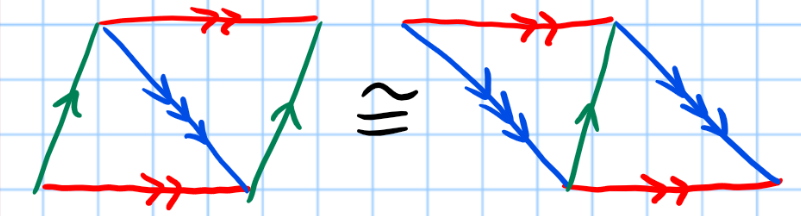}\caption{Two parallelograms yielding the same flat torus}
\end{figure}
\begin{thm}[Flat tori]\label{t:flat}
Every flat torus is obtained, up to isometry and homothety, from exactly one parallelogram with vertices $(0,0)$, $(1,0)$, $(x_0,y_0)$ and $(x_0+1,y_0)$, where
\[
(x_0,y_0)\in\Gamma:=\{x^2+y^2\geq 1,\ 0\leq x\leq 1/2,\ y\geq0 \}.
\]
\end{thm}
\medskip

\noindent{\!
\begin{minipage}[t]{.58\textwidth}
\begin{ex}
Prove the theorem. Hint: Use that a flat torus is obtained as the quotient $\R^2/G$ for some subgroup of translations $G$ generated by $p\mapsto p+v_1$ and $p\mapsto p+v_2$, where $v_1,v_2\in\R^2$ are linearly independent. Without loss of generality you can assume that $v_1$ is an element of minimal norm in $G\setminus 0$. Then, show that $v_2$ can be chosen to lie in $\Gamma$ (existence). Finally, prove that $v_2$ obtained in this way is the shortest element in $G$ which is not a multiple of $v_1$ (uniqueness).
\end{ex}	
\end{minipage}
\begin{minipage}[t]{.42\textwidth}
	\kern-7pt\centering\includegraphics[width=.57\textwidth]{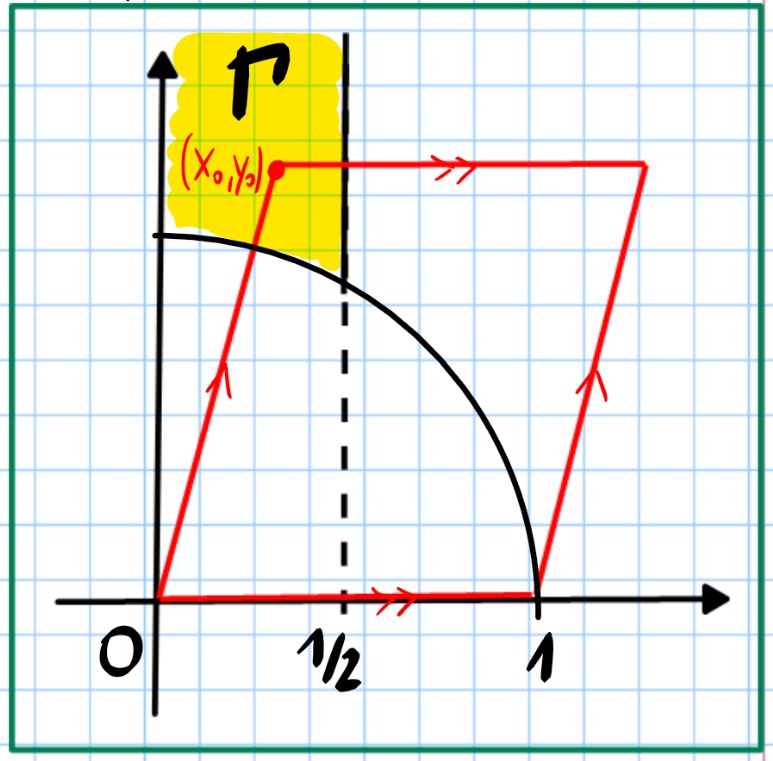}
	\captionof{figure}{The domain $\Gamma$}
\end{minipage}
}
\medskip

On $K^2$ flat metrics are obtained by identifying the opposite side of a rectangle, where the horizontal sides are identified by a translation and the vertical sides are identified by a reflection through their midpoint. Thus flat metrics on $K^2$ up to isometry and homothety are parametrized by the positive real number which is the ratio between the length of the horizontal and the vertical side of the rectangle.
\begin{ex}
Prove the classification of flat Klein bottles. Hint: A flat Klein bottle is obtained as the quotient $\R^2/G$, where $G$ is the subgroup of isometries of $\R^2$ generated by a translation of vector $v_1$ and a glide reflection which is composition of a reflection with respect a line and a translation by a vector $v_2$ parallel to that line. All elements of $G$ are either translations or glide reflections. Without loss of generality consider the non-zero translation with minimal $|v_1|$ and the glide reflection with minimal $|v_2|$ and show that $v_1$ and $v_2$ are orthogonal. 	
\end{ex}
If we now consider surfaces of higher genus $\Sigma_k$ or $\tilde\Sigma_k$, metrics of constant curvature have negative curvature (so called hyperbolic metrics) and classifying them up to isometry is a highly non-trivial task. This has strong relationships with complex analysis and to the so-called moduli space of Riemann surfaces (see \cite{FM12} for more information).

Complex analysis is indeed a powerful tool to study the geometry of surfaces and gives us a normal form for an arbitrary metric on an arbitrary surface, which follows from the celebrated Riemann's uniformization theorem.
\begin{thm}[Uniformization]
Every Riemannian metric $g$ on a surface is conformally equivalent to a metric $g_*$ of constant curvature. In other words, $g$ is isometric to $f^2g_*$, where $f$ is a smooth positive function on the surface.
\end{thm}
\begin{rmk}
The adjective conformal refers to the fact that $f^2g_*$ and $g_*$ measure the same angles although for $f\neq1$ do not measure the same lengths.	
\end{rmk}
\subsection{Systoles}
A powerful way to study $g\in\mathcal R(M)$ is via its geodesics $\gamma:\R\to M$. From a geometric point of view, these are the curves with constant speeds which locally minimize the length. From an analytical point of view they are curves with zero acceleration, namely they satisfy the second-order ODE $\nabla_{\dot\gamma}\dot\gamma=0$, where $\nabla$ is the Levi-Civita connection of $g$.
\begin{thm}[Cartan 1928]
If $M$ is not simply connected and $g\in\mathcal R(M)$, then there exists a non-contractible loop of minimal length and any such loop is a periodic geodesic.
\end{thm}
\begin{dfn}
We call systole of $g\in\mathcal R(M)$ a non-contractible loop of minimal length and we denote by $\mathrm{sys}(g)$ its length.
\end{dfn}
\begin{ex}
Show that a systole has no self-intersection.	
\end{ex}
\begin{ex}\label{ex:closed}
Let $S\subset M$ be the set of points $p$ such that there exists a systole passing through $p$. Show that $S$ is a closed subset of $M$. Hint: use that systoles are geodesics and the continuous dependence of solutions of ODE from the initial data.
\end{ex}
What can be said about $\mathrm{sys}(g)$? From the pictures of tori, we see that we can have a big torus with a short systole. From the pictures is however not clear if we can have a "small" torus with a long systole. Here we have to quantify what do we mean by a "big" or "small" manifold.  For example, one could use the diameter $\mathrm{diam}_g(M)$, namely the maximum distance between two points on $M$, to measure the size of a manifold. In this case, one has $\mathrm{sys}(g)\leq 2\mathrm{diam}_g(M)$, so that long systoles exist only on big manifolds.
\begin{ex}
Show the inequality	$\mathrm{sys}(g)\leq 2\mathrm{diam}_g(M)$. Hint: Let $s:=\mathrm{sys}(g)$ and let $\gamma:[0,s]\to M$ with $\gamma(0)=\gamma(s)$ be a systole parametrized with unit speed. Consider a length minimizing curve $\delta$ from $\gamma(0)$ to $\gamma(s/2)$.
\end{ex}
Another natural way to measure the size of a manifold is through its total volume $\vol_{\!g}(M)$ with respect to $g$.  This brings us to ask if the so-called \textit{systolic inequality} holds.
\begin{que}
Let $M$ be not simply connected. Does there exist a positive constant $C$ depending on $M$ such that 
\[
\sigma(g):=\frac{\mathrm{sys}(g)^m}{\vol_{\!g}(M)}\leq C,\qquad\forall\,g\in\mathcal R(M),
\]
where $\sigma(g)$ is the so-called systolic ratio. If $C$ exists, what is its optimal value? Are there metrics maximizing $\sigma$? How do they look like?
\end{que}
\begin{ex}\label{ex:all}
Let $g\in\mathcal R(M)$ be a metric maximizing $\sigma$. Show that there is a systole through every point of $M$. Hint: Let $S$ be the set desciribed in Exercise \ref{ex:closed}. If $M\setminus S$ is non-empty, then we can construct a metric $g_1=fg$, where $f$ is a function supported in $M\setminus S$ with smaller volume and the same systolic length as $g$.
\end{ex}
\subsection{Systolic inequality on surfaces} The first systolic result is the complete answer to the above question on $T^2$. It was proved by Charles Loewner during his lectures on Riemannian Geometry at the University of Syracuse in 1949 \cite{Kat07, Pu52}. Loewner never published this result which appears for the first time in written form in the doctoral thesis of Pao Ming Pu.
\begin{thm}[Loewner 1949 \cite{Pu52}]\label{t:loewner}
There holds
\[
\sigma(g)\leq \frac{2}{\sqrt{3}},\qquad\forall\,g\in\mathcal R(T^2). 
\]
Equality holds if and only if, up to homothety and isometry, $g$ is the flat metric corresponding to the parallelogram with equal sides making an angle of $60$ degrees.
\end{thm}
\begin{proof}
Since the systolic ratio is invariant by isometries we can suppose that $g=f^2g_*$ for some $g_*$ flat obtained from a parallelogram as in Theorem \ref{t:flat}. Let us write $(x,y)=r(\cos\alpha,\sin\alpha)$ for some $r\geq 1$ and $\alpha\in[\pi/3,\pi/2]$. For each $t\in[0,1]$, consider the horizontal curve $s\mapsto\gamma_t(s):=(s+tx,ty)$, $s\in[0,1]$. Each $\gamma_t$ is a systole of $g_*$ with length $1$. Therefore, we get
\begin{align*}
\vol_{\!g}(T^2)&=\int_0^{1}\int_0^1f^2(s+tx,ty)r\sin\alpha\,\dd s\,\dd t\\
&=r\sin\alpha\int_0^{1}\Big(\int_0^1f^2(s+tx,ty)\dd s\Big)\dd t\\
&\geq \frac{1\cdot r\sin\alpha}{1^2}\int_0^{1}\Big(\int_0^1f(s+tx,ty)\dd s\Big)^2\dd t\qquad\qquad\text{\small (Arithmetic Mean-Quadratic Mean)}\\
&=\frac{\vol_{\!g_*}(T^2)}{\mathrm{sys}(g_*)^2}\int_0^{1}\ell_g(\gamma_t)^2\dd t\\
&\geq \frac{1}{\sigma(g_*)}\min_{t\in[0,1]}\ell_g(\gamma_t). \hspace{140pt}\text{\small (Minimum-Arithmetic Mean)}\\
&\geq \frac{1}{\sigma(g_*)}\mathrm{sys}(g)^2. \hspace{163pt}\text{\small (definition of $\mathrm{sys}(g)$)}
\end{align*}
The first inequality is an equality if and only if $f$ does not depend on $s$ and the second inequality is an equality if and only if $f$ is constant. Rearranging terms we find $\sigma(g)\leq \sigma(g_*)$ with equality if and only if $g$ is homothetic to $g_*$. Finally, since $\mathrm{sys}(g_*)=1$, the systolic ratio $\sigma(g_*)$ is maximum when $\vol_{\!g_*}(T^2)$ is minimum, which happens if and only if $r=1$ and $\alpha=\pi/3$.  
\end{proof}
\begin{rmk}[Important]\label{r:important}
	Summing up, the sharp systolic for $T^2$ relies on the following two steps:
	\begin{itemize}
		\item Write $g$ in a normal form (uniformization theorem);
		\item Use the inequality between the minimum, the arithmetic mean and the quadratic mean on a set of loops which foliate the space.
	\end{itemize}
	We will see in Section \ref{ss:final} of Lecture 2 that the local systolic inequality for contact manifolds follows a similar scheme.
\end{rmk} 
In his doctoral thesis, Pu pushed Loewner's ideas further proving the sharp systolic inequality for the real projective plane.
\begin{thm}[Pu 1952 \cite{Pu52}]
There holds
\[
\sigma(g)\leq \frac{\pi}{2},\qquad\forall\,g\in\mathcal R(P^2). 
\]Equality holds if and only if $g$ is a metric of constant curvature.
\end{thm}

\begin{rmk}
The proof of Pu's theorem also starts with the uniformization theorem and then constructs a metric with larger systolic ratio by averaging on the group of isometries of the metric $g_*$ of constant curvature on $P^2$. By the transitivity of this group, the averaged metric is a multiple of $g_*$ and, hence, we get the desired result. A beautiful, simple argument following more closely the proof of Loewner's theorem we presented above is given in \cite{KN20}, where also a formula for the remainder $\vol_{\!g}(P^2)-\tfrac{2}{\pi}\mathrm{sys}(g)^2$ in terms of the conformal factor is computed. There, the $g$-area is also obtained by averaging the energy of a certain family of loops, and then again the Minimum-Arithmetic Mean-Quadratic Mean inequality is used. However, instead of a one-parameter family of closed loops (on the flat torus, we took all horizontal geodesics), the two-parameter family consisting of all geodesics of the metric $g_*$ has to be used.  
\end{rmk}
Using a refinement of the argument of Loewner and Pu, the sharp systolic inequality for the Klein bottle can also be established. However, in this case the maximizing metric is singular.
\begin{thm}[Bavard 1986 \cite{Bav86}]
There holds
\[
\sigma(g)\leq \frac{\pi}{2\sqrt{2}},\qquad \forall\,g\in\mathcal R(K^2). 
\]
There is no maximizing smooth metric but there is a unique singular maximizing metric $g_*$ which is obtained as $C^0$-limit of smooth ones. The metric $g_*$ is constructed by gluing together two Möbius bands obtained as follows: Cut the region on the northern hemisphere of $S^2$ which is at distance $\pi/4$ from the half great circle perpendicular to the boundary of the hemisphere; glue the two sides lying on the boundary of the hemisphere via the antipodal map.
\end{thm}
For higher genus surfaces $\Sigma_k$, the systolic inequality was first proved by Accola and Blatter in 1960 using again the uniformization theorem as main tool. They got a constant $C_k$ growing with the genus $k$. However, one would expect to get a constant decreasing with the genus since if we take $\mathrm{sys}(g)=1$, then we are attaching $k$ handles which have waists of length at least $1$ and cannot be too short. Thus, the volume should increase with the genus (don't take this argument too seriously). The systolic inequality with the sharp behaviour of $C_k$ in the genus $k$ is a deep result established by Gromov in the Eighties.
\begin{thm}[Gromov 1983 \cite{Gro83}]\label{t:hyp}
There exists $C>0$ such that for every $k\geq 2$ there holds
\[
\sigma(g)\leq \frac{(\log k)^2}{k}C,\qquad\forall\,g\in\mathcal R(\Sigma_k).
\]
The same statement holds true also on the non-orientable surface $\tilde \Sigma_k$ for $k\geq2$.
\end{thm}
\begin{ex}
Deduce Gromov's theorem for $\tilde\Sigma_k$ from Gromov's theorem for $\Sigma_k$ with $k\geq2$.
\end{ex}
\begin{ex}
Show that a metric $g\in\Sigma_k$ with constant negative curvature cannot maximize the systolic ratio. Hint: From Riemannian geometry, there is only one periodic geodesic in each free-homotopy class since $g$ has negative curvature. Deduce that there are only finitely many systoles and conclude using Exercise \ref{ex:all}.
\end{ex}
\subsection{Systolic inequality in higher dimension}
Let us now briefly discuss the systolic inequality in dimension $m\geq3$. Gromov's result generalizes to hyperbolic manifolds, namely those manifolds such that there exists $g_*\in\mathcal R(M)$ with constant curvature $-1$. For these manifolds $V(M):=\vol_{\!g_*}(M)$ is a topological invariant thanks to Mostow rigidity theorem (for hyperbolic surfaces this is also a topological invariant by Gauss--Bonnet: $V(\Sigma_k)=2(k-1)$). Gromov's generalization of Theorem \ref{t:hyp} states that for every $m$ there is a constant $C$ such that for every hyperbolic manifold of dimension $M$:
\[
\sigma(g)\leq \frac{(\log V(M))^2}{V(M)}C,\qquad\forall\,g\in\mathcal R(M). 
\]
On the other hand, there are non-simply connected manifolds for which the systolic inequality does not hold. The simplest of which is $M=S^1\times S^2$ as one sees by taking a product metric on $M$ with suitably scaled factors. 
\begin{ex}
Justify the above assertion.
\end{ex}
From this example we learn that the systolic inequality is expected only for those manifolds having non-contractible loops in every direction (on $S^1\times S^2$ there is no such a loop along the $S^2$-direction). Intuitively, real projective spaces $P^m$, tori $T^m$ and, more generally, manifolds whose universal cover is $\R^m$ have this property. Gromov formalized this intuition via the following precise definition and proved the subsequent amazing result.
\begin{dfn}
A manifold $M$ is called essential if the map $H_m(M)\to H_m(K(\pi_1(M),1))$ is non-zero, where $K(\pi_1(M),1)$ is the corresponding Eilenberg-Maclane space.
\end{dfn}
\begin{thm}[Gromov 1983 \cite{Gro83}]
For every $m$, there exists a constant $C$ such that for all essential manifolds $M$ of dimension $m$, there holds
\[
\sigma(g)\leq C,\qquad \forall\,g\in\mathcal R(M).
\]	
\end{thm}
This result is saying that the systolic inequality holds for every essential manifold with a constant depending only on the dimension. There is however no example in dimension $m\geq 3$, where the optimal constant of a given manifold is known. It is also absolutely remarkable that Gromov's result is sharp for orientable manifolds.
\begin{thm}[Babenko 1992 \cite{Bab92}]
If $M$ is orientable and satisfies the systolic inequality, then $M$ is essential.
\end{thm}
\subsection{Systoles on simply connected manifolds?}
The notion of systole makes sense only for metrics on non-simply connected manifolds $M$ since we are considering shortest non-contractible loops. Can one define a similar concept on simply connected manifolds? By Cartan's theorem, the systole is also the length of the shortest non-contractible \textit{geodesic}, thus we can give the following natural definition.
\begin{dfn}
We define the weak systole\footnote{This is not the standard terminology in Riemannian geometry but I will use it here to make it (weakly) consistent with the terminology in symplectic geometry.} of a metric $g\in\mathcal R(M)$ as a (non-constant) periodic geodesic for $g$, whose length is minimal \textit{among all (non-constant) periodic geodesics} for $g$. We denote by $\widetilde{\mathrm{sys}}(g)$ the length of a weak systole and define the weak systolic ratio
\[
\rho(g):=\frac{\widetilde{\mathrm{sys}}(g)^m}{\vol_{\!g}(M)},
\]
which clearly satisfies $\rho(g)\leq \sigma(g)$ for all $g\in\mathcal R(M)$.
\end{dfn}
Can one prove a (weak) systolic inequality for $\rho$ on a simply connected manifold $M$? The first thing we need to check is that $\widetilde{\mathrm{sys}}(g)$ is finite for all $g\in\mathcal R(M)$, which is equivalent to the existence of periodic geodesics for $g$, a natural generalization of Cartan's theorem.
\begin{thm}[Birkhoff 1917 \cite{Bir17} for $S^2$; Lusternik--Fet 1951 \cite{LF51} in general]
If $M$ is simply connected, then every $g\in\mathcal R(M)$ admits a non-constant periodic geodesic.
\end{thm}

\noindent\!\begin{minipage}{.5\textwidth}
\begin{rmk}
The periodic geodesic found in the theorem cannot be obtained by minimizing the length in the space of loops since the minimum is attained at constant loops. One needs instead to use a minimax argument. For $S^2$ this argument is due to Birkhoff and amounts to considering $1$-parameter families of loops starting and ending at constant loops with the property that the associated map $F:S^2\to S^2$ has degree $1$ as in the picture below. Then, one starts homotoping the map $F$ via the gradient flow of the length functional in order to decrease the length of all loops in the family simultaneously. Since $F$ is not homotopic to a constant at least one of the loops will stay with length bounded away from zero during the deformation and therefore will converge to a periodic geodesic with positive length. 
\end{rmk}
\end{minipage}
\begin{minipage}{.5\textwidth}
\centering\includegraphics[width=.5\textwidth]{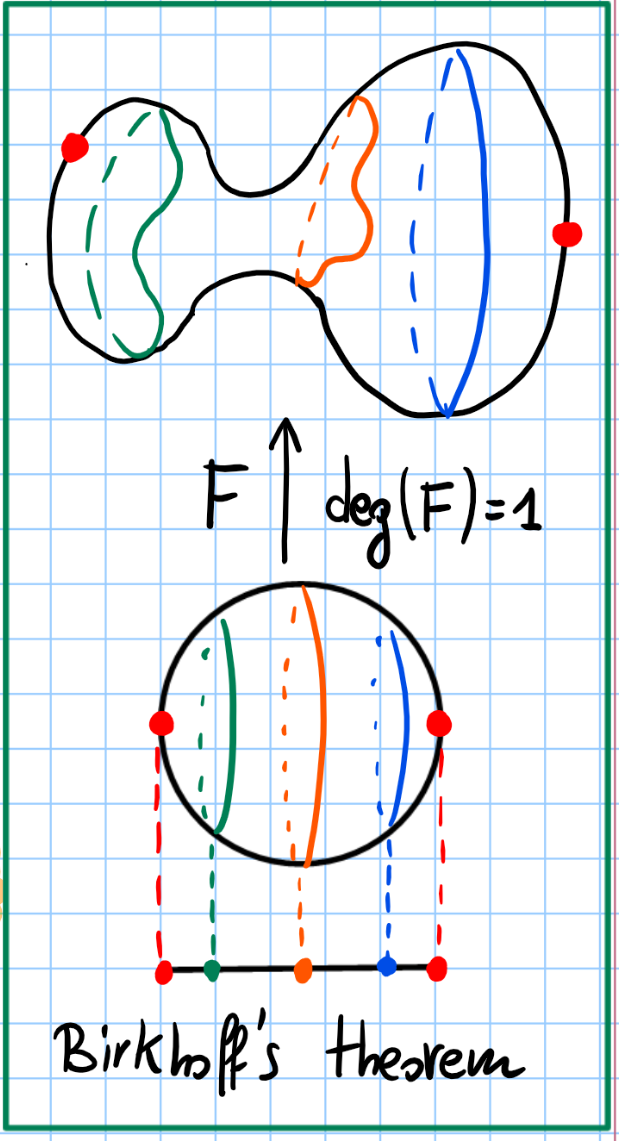}\captionof{figure}{Minimax on $S^2$}	
\end{minipage}
\medskip

There is, however, only one case where the weak systolic inequality on simply connected manifolds has been established, namely on $S^2$, the unique surface where the systolic inequality does not hold.
\begin{thm}[Croke 1988 \cite{Cro88}; Rotman 2006 \cite{Rot06}]
There exists $C>0$ such that
\[
\rho(g)\leq C,\qquad \forall\,g\in \mathcal R(S^2). 
\]
The optimal constant $C$ lies in the interval $[2\sqrt{3},2\sqrt{8})$.
\end{thm}
It is conjectured that the optimal constant is $2\sqrt{3}$ and it is achieved exactly by the so-called Calabi--Croke metric $g_{\Delta}$ which is obtained by gluing two equilateral triangles along the corresponding edges (what is the systole in this case?). This conjecture is compatible with the fact that smooth metrics close to $g_\Delta$ have smaller weak systolic ratio \cite{Bal10,Sab10}.

The theorem above tells us in particular that the round metric $g_*$ on $S^2$ does not maximize the weak systolic ratio since
\[
\rho(g_*)=\pi<2\sqrt{3}=\rho(g_{\Delta}).
\]
This should be surprising since we know from Pu's theorem that the round metric on $P^2$ is the unique maximizer for the systolic ratio on $\mathcal R(P^2)$. Is then $g_*$ at least a strict local maximizer for the systolic ratio on $\mathcal R(S^2)$? As we will see, the answer to this question is \textit{almost} affirmative and it will give us the opportunity to introduce another important class of Riemannian metrics.

\subsection{Zoll metrics}\label{s:Zoll}
The round metric on $S^2$ or $P^2$ has the property that all geodesics are periodic and with the same length. Let us give a name to metrics having this property on arbitrary manifolds.

\begin{dfn}
A metric $g$ on a manifold $M$ is said to be Zoll if all geodesics of $g$ are periodic and with the same length. Let us denote by $\mathcal Z(M)\subset\mathcal R(M)$ the (possibly empty) subset of Zoll metrics on $M$.
\end{dfn}
\begin{rmk}
One can equivalently define Zoll metrics by requiring that every geodesic on $M$ is a weak systole (why?).
\end{rmk}
The reason for the name of these particular metrics comes from Otto Zoll, who was the first to construct metrics with the above property on $S^2$ different from the round one \cite{Zol03}. More precisely, he gave an infinite dimensional family of pairwise non-isometric examples having the above property among metrics coming from spheres of revolutions in $\R^3$. In spherical coordinates $(\theta,\phi)$ they are exactly of the form 
\[
g=(1+h(\cos\theta))^2\dd\theta^2+(\sin\theta)^2\dd\phi^2,
\]
where $h:[-1,1]\to(-1,1)$ is any odd function with $h(-1)=0=h(1)$.

Still nowadays global properties of the space $\mathcal Z(S^2)$ remain quite mysterious and we do not even know if it is connected. However, we have the following remarkable result, which describes the local structure of $\mathcal Z(S^2)$ around $g_*$, up to the uniformization theorem.

\begin{thm}[Funk 1913 \cite{Fun13}; Guillemin 1976 \cite{Gui76}]
Let $s\mapsto f_s\colon S^2\to(0,\infty)$ be a smooth one-parameter family of functions with $s\in(-\epsilon,\epsilon)$ and $f_0=1$. If $f_s^2g_*$ is Zoll for every $s$, then $\phi:=\partial_s|_{s=0} f_s$ is an odd function, namely $\phi(-x)=-\phi(x)$ for every $x\in S^2$ (here $-x$ denotes the antipodal point of $x$ on $S^2$). Conversely, given an odd function $\phi:S^2\to\R$, there exists a one-parameter family $s\mapsto f_s$ such that $\phi=\partial_s|_{s=0} f_s$ and $f_s^2g_*$ is Zoll for every $s\in(-\epsilon,\epsilon)$.
\end{thm}
The round metric $g_*$ on $S^2$ descends to a Zoll metric on $P^2$ since the antipodal map is an isometry. However, none of the Zoll metrics in the theorem above descends to $P^2$ since a metric $fg_*$ is invariant under the antipodal map when the function $f$ is \textit{even}. Indeed, much more is true.
\begin{thm}[Green 1963 \cite{Gre63}]
The metric $g_*$ with constant curvature is the only Zoll metric on $P^2$: $\mathcal Z(P^2)=\{g_*\}$.
\end{thm}
\begin{rmk}
The original proof of the theorem is based on the characterization of \textrm{Wiedersehen} metrics on $S^2$. These are the metrics such that for every $x\in S^2$ there exists $\mathcal C(x)\in S^2$ such that all geodesics starting from $x$ have $\mathcal C(x)$ as first conjugate point. This roughly means that all geodesics starting from $x$ will converge again in the point $\mathcal C(x)$. Green proved that the only Wiedersehen metric on $S^2$ is the round metric and thus the theorem follows by showing that the lift of a Zoll metric on $P^2$ to $S^2$ is a Wiedersehen metric. For more about Wiedersehen metrics have a look at the notes by Ballmann \cite{Bal16} and at Chapter 5 in Besse's book \cite{Bes78}, where they are put in the context of Blaschke manifolds. 
\end{rmk}
\begin{ex}
Prove the statement above: The lift of a Zoll metric on $P^2$ to $S^2$ is a wiedersehen metric. Hint: Show that each periodic geodesic on $P^2$ lifts to a minimizing geodesic on $S^2$ connecting a pair of antipodal points.
\end{ex}
We have discussed Zoll metrics on $S^2$ and $P^2$. Are there Zoll metrics on $T^2$ or on the other surfaces? The answer is no. Indeed, if $\mathcal Z(M)\neq\varnothing$, then either $M$ is simply connected or $\pi_1(M)\cong\Z/2\Z$ \cite{MO13}.
\begin{ex}
Prove the last statement: If $g\in\mathcal Z(M)$ and $M$ is not simply connected, then $\pi_1(M)$ has only one non-trivial element. Hint: Fix $x\in M$ and let $\gamma$ be any systole passing through $x$. The homotopy class $a\in\pi_1(M,x)$ of $\gamma$ does not depend on the systole we chose (why?). Let now $b$ be any other non-trivial homotopy class in $\pi_1(M,x)$. By the Hopf-Rinow theorem applied to the universal cover of $M$, there exists a geodesic loop $\delta$ based at $x$ in the homotopy class $b$. Show that $\delta$ is an iteration of a systole $\gamma$ based at $x$ (Sub-hint: Take $\gamma$ tangent to $\delta$ and use that systoles do not have self-intersections). This proves that $b=a^k$ for some $k\geq1$. Thus, $\pi_1(M,x)$ is a finite (take $b=a^{-1}$), cyclic group generated by $a$. To finish the argument, show that $a=a^{-1}$. Hint: reverse the parametrization of $\gamma$.
\end{ex}
\begin{rmk}
In higher dimensions metrics of constant curvature on $S^m,P^m$ and the standard metrics on complex $P^d_{\C}$ and quaternionic $P^d_{\mathbb H}$ projective space and on the Cayley plane $P^2_{\mathrm{Ca}}$ are Zoll. It is conjectured that these are the only manifolds admitting Zoll metrics (it is known that the cohomology ring must be the same as one of these \cite{Bes78}). Non-trivial examples of Zoll metrics are only known on $S^m$ for every $m$. Using the same argument as for $P^2$, Green's theorem generalizes to higher dimension: The only Zoll metric on $P^m$ is the standard one for every $m$.
\end{rmk}
\subsection{Zoll metrics and the weak systolic ratio on the two-sphere} By the results of Zoll and Guillemin, there are many Zoll metrics on $S^2$. However, they all have the same weak systolic ratio.
\begin{thm}[Weinstein 1974 \cite{Wei74}]\label{t:wei}
If $g\in\mathcal Z(S^2)$, then $\rho(g)=\pi$.
\end{thm}
\begin{ex}
Using Weinstein's theorem show that $\rho(g)=\pi/2$ for every $g\in\mathcal Z(P^2)$ by lifting the metric to $S^2$. Combine this result with Pu's theorem to show Green's theorem above. 	
\end{ex}
This shows that the round metric $g_*$ cannot even be a strict local maximizer for $\rho$ since we can deform $g_*$ to a non-trivial Zoll metric. However, this is the only thing that can go wrong.
\begin{thm}[Abbondandolo--Bramham--Hryniewicz--Salom\~ao 2018 \cite{ABHS18}]\label{t:abhs18}
Let $g_0$ be a Zoll metric on $S^2$. If $g$ is a metric on $S^2$ sufficiently $C^3$-close to $g_0$, then
\[
\rho(g)\leq \pi\quad\text{and}\quad\Big(\rho(g)=\pi\iff g \textrm{ is Zoll}\,\Big).
\]
\end{thm}
Motivated by this result one can look for class of metrics on $S^2$ whose systolic ratio is bounded by $\pi$. Such a class is given by metrics which have a rotational symmetry as those coming from spheres of revolution in $\R^3$.
\begin{thm}[Abbondandolo--Bramham--Hryniewicz--Salom\~ao 2018 \cite{ABHS18c}]
If $g$ is a rotationally symmetric metric on $S^2$, then 
\[
\rho(g)\leq \pi\quad\text{and}\quad\Big(\rho(g)=\pi\iff g \textrm{ is Zoll}\,\Big).
\]
\end{thm}
We remark, finally, that the local systolic inequality in a neighborhood of a Zoll metric is not only a two-dimensional phenomenon but holds for Zoll metrics in any dimension.
\begin{thm}[Abbondandolo--Benedetti 2019 \cite{AB19}]\label{t:riemzoll}
	Let $g_0$ be a Zoll metric on a manifold $M$. If $g$ is a metric on $M$ sufficiently $C^3$-close to $g_0$, then
	\[
	\rho(g)\leq \rho(g_0)\quad\text{and}\quad\Big(\rho(g)=\rho(g_0)\iff g \textrm{ is Zoll}\,\Big).
	\]
\end{thm}
The results of Weinstein and of Abbondandolo, Bramham, Hryniewicz, Salom\~ao and of Abbondandolo, Benedetti are based on techniques from Symplectic Geometry which we will explore in the next talk. Geodesics are curves on $M$ with zero acceleration. They satisfy therefore a second-order differential equation. In Euclidean space there is a standard procedure to pass from a second-order equation $\ddot x=f(x,\dot x)$ to a first-order equation $\dot X=F(X)$ at the expense of doubling the number of variables. Indeed, we can include the velocity components as a variable by defining $X=(x,v)$ and $F(X)=(v,f(x,v))$. Then, there is a bijection between the solutions of $\ddot x=f(x,\dot x)$ and of $\dot X=F(X)$. For geodesics on a general manifold $M$, this procedure amounts to considering a first-order equation on the tangent space $TM$, which will yield a Hamiltonian system with respect to a canonical symplectic form, upon identifying $TM$ with the cotangent bundle $T^*M$.
\newpage

\section{Second talk: Systolic inequalities in Symplectic geometry}
Goal of this talk is to discuss how symplectic geometry can help us in studying the local systolic inequality in a neighborhood of a Zoll metric. The key is due to the fact that there are many more transformations on $T^*M$ preserving the symplectic structure (so-called symplectomorphisms) than transformations on $M$ preserving the metric (so-called isometries). Therefore, we have more chances to use this larger group of transformations to find a useful normal form for our problem. If you recall the central role played by the uniformization theorem (which essentially gives us a normal form for the metric) in proving the systolic inequalities on $T^2, P^2$, you can see how crucial normal forms can be.

This circle of ideas works only because the weak systolic ratio $\rho$ is invariant under symplectic transformations. This remarkable fact will bring us to formulate, in the initial part of this talk, a natural version of the systolic inequality first on $\R^{2n}$ and then on more general symplectic manifolds. Unfortunately, these inequalities never hold globally \cite{ABHS19,Sag18}. The reason of this global failure is the same as the reason of the local success. The world of symplectic manifolds is much more flexible than the world of Riemannian manifolds, so it is easier to construct wild counterexamples. This is extremely interesting because then one can ask where the boundary between local success and global failure lies.

We will see that assuming convexity conditions will be enough to get a systolic inequality in many cases:
\begin{itemize}
\item On cotangent bundles this amounts to considering the class of Finsler metrics, a generalized version of Riemannian metrics where the notion of length is well-defined but the notion of angle is not. The rule of thumb is that when the systolic inequality holds for Riemannian metrics on $M$, then it holds also for Finsler metrics on $M$.
\item On the simplest symplectic manifold $\R^{2n}$ this will bring us to consider the Viterbo conjecture regarding the optimal constant for the systolic inequality on convex sets \cite{Vit00}.
\end{itemize}
It is important to notice, nonetheless, that convexity is not an invariant property under symplectic transformations. Therefore, it is still an interesting open problem to find a natural symplectic setting for the systolic inequality (although there are some interesting candidates \cite{ABHS18b}).

\subsection{Hamiltonian systems in euclidean space}
Every smooth function $H:\R^{n}\times\R^n\to\R$ yields a corresponding Hamilton equation for curves $t\mapsto (q(t),p(t))\in\R^n\times\R^n$. This is the first-order differential equation given by
\[
\left\{\begin{matrix}
\displaystyle\dot q(t)=+\frac{\partial H}{\partial p}(q(t),p(t))\vspace{7pt}\\ 
\displaystyle\dot p(t)=-\frac{\partial H}{\partial q}(q(t),p(t))
\end{matrix}
\right.
\]
As every first-order ordinary differential equation, the dynamics is described by the flow maps $(q,p)\mapsto\Phi^t(q,p)$ yielding the position after time $t\in\R$ of a solution starting at $(q,p)$ at time $0$. Flows satisfy the group property: $\Phi^0=\mathrm{id}_{\R^{n}\times\R^n}$ and $\Phi^{s+t}=\Phi^s\circ\Phi^t$.

The function $H$ is called the Hamiltonian of the system and is a conserved quantity along the motion which we interpret as the energy of the system.
\begin{ex}
Prove the last statement by showing that $\tfrac{\dd}{\dd t}H\circ (q(t),p(t))=0$.
\end{ex}
Thus for every $h\in\R$, the hypersurface $S_h:=H^{-1}(h)$ is invariant under the flow, which means that if a solution is contained in $S_h$ at some time $t$, it remains in $S_h$ for all times.

Important examples are given by systems with a mechanical Hamiltonian
\[
H(q,p)=\tfrac{1}{2}|p|^2+U(q),
\]
where $U:\R^n\to\R$ is a smooth function. In this case $H$ represents the sum of the kinetic and the potential energy of a particle with unit mass and Hamilton's equation are equivalent to Newton's second law for the particle $\ddot q(t)=-\nabla U(q(t))$, where $\nabla$ is the gradient of $U$.
\begin{rmk}\label{r:harm}
When
\[
U(q)=\frac12\sum_{j=1}^n\frac{1}{a_j^2}q_j^2,
\]
the mechanical Hamiltonian describes a system of $n$ uncoupled harmonic oscillators with frequencies $2\pi a_i$. In particular, if all the numbers $a_i$ are equal all orbits are periodic with the same period. Finally, notice that the energy levels of the harmonic oscillators are homothetic ellipses and they are spheres exactly when the $a_i$ are all equal.
\end{rmk}
\subsection{Enters the symplectic form}
Symplectic geometry can be understood as the geometry of classical mechanics, or more generally of Hamiltonian systems. According to the ideas of Felix Klein in his \textit{Erlangensprogramm}, doing geometry means considering a group of transformations $G$ acting on a space $E$, which we interpret as symmetries, and classify subsets of the space up to the action of this group. A classical example is euclidean geometry of the plane, where $E=\R^2$ and $G$ is the group of euclidean isometries, and one can classify triangles in $\R^2$ up to isometry. On the same space, there might be several group acting and one can compare the resulting geometries. For instance, if we consider the group of affine transformations on the plane we get a geometry which is different from the euclidean one, since up to affine transformations there is only one triangle but there are many different triangles up to isometries. In this case, one says that euclidean geometry is more rigid than affine geometry, or, equivalently, that affine geometry is more flexible than euclidean geometry.

In symplectic geometry, we study those transformations $(q,p)\mapsto (Q,P)$ of $\R^n\times\R^n$ which leave the form of Hamilton equations invariant meaning that for all Hamiltonians $H$, there holds
\[
\left\{\begin{matrix}
	\displaystyle\dot q=+\frac{\partial H}{\partial p}(q,p)\vspace{7pt}\\ 
	\displaystyle\dot p=-\frac{\partial H}{\partial q}(q,p)
\end{matrix}
\right.\quad\iff\quad \left\{\begin{matrix}
	\displaystyle\dot Q=+\frac{\partial H}{\partial P}(Q,P)\vspace{7pt}\\ 
	\displaystyle\dot P=-\frac{\partial H}{\partial Q}(Q,P)
\end{matrix}
\right.
\]
\vspace{2pt}

Said in other words: The change of coordinates $(q,p)\to (Q,P)$ sends solutions to the Hamilton equations of $H$ to solutions to the Hamilton equations of the function $H$ expressed in the new coordinates, and vice versa.
We call $\mathrm{Symp}(\R^{2n})$ the group of transformations having this properties. They are referred to as canonical transformations, or, in more modern terms, to symplectomorphisms. The above definition of symplectomorphism is not the easiest to work with since we need to check that the Hamilton equations are preserved for all functions $H$. To get a better description, we rewrite Hamilton's equations in a coordinate-free way.

We define the following non-degenerate two-form on $\R^{2n}$:
\[
\omega_0:=\sum_{j=1}^n\dd p_j\wedge \dd q_j.
\]
Non-degenerate means that, if $(\R^{2n})^*$ denotes the dual of $\R^{2n}$, then the map $\R^{2n}\to (\R^{2n})^*$, $u\mapsto\omega_0(u,\cdot)$ is an isomorphism. Therefore, for every function $H$, we can uniquely define a vector field $X_H$ on $\R^{2n}$ via the relation
\[
\omega_0(X_H,\cdot)=-\dd H,
\]
where $\dd H$ denotes the differential of $H$. The vector field has the block form
\[
X_H=\begin{pmatrix}
\displaystyle+\frac{\partial H}{\partial p}\vspace{5pt}\\ \displaystyle-\frac{\partial H}{\partial q}.	
\end{pmatrix}
\]
so that Hamilton's equation can be rewritten as $\dot z=X_H(z)$. Using the transformation rules for vector fields, forms and functions, we arrive at the following result.
\begin{thm}
A transformation $\varphi$ is a symplectomorphism if and only if $\varphi$ preserves the symplectic form: $\varphi^*\omega_0=\omega_0$.
\end{thm}
\begin{ex}\label{ex:un}
Identify $\R^{2n}$ with $\C^n$ via $z_j=q_j+ip_j$ for $j=1,\ldots,n$. Show that $\omega_0$ is the imaginary part of the standard hermitian product $h=\sum_{j=1}^n\dd z_j\dd\bar z_j$. Deduce that the unitary group $U(n)$ is a subgroup of $\mathrm{Symp}(\R^{2n})$. Recall that a map is unitary if it is linear and preserves $h$.
\end{ex}
At this point, we naturally ask how big is the group of symplectomorphisms. It is actually an infinite dimensional group, whose one-dimensional subgroups are exactly given by Hamiltonian flows.
\begin{thm}
For every $H$ and for every $t$ such that $\Phi^t$ is defined on the whole $\R^{2n}$, the flow map $\Phi^t$ is a symplectomorphism.
\end{thm}
\begin{ex}
Prove the theorem completing the details below. The condition that $\Phi^t$ is a symplectomorphism can be rewritten as
\[
\omega_0(\dd\Phi^t\cdot u,\dd\Phi^t\cdot v)=\omega_0,\qquad\forall\,t\in\R,\qquad\forall\,u,v\in\R^{2n}.
\]
Since $\Phi^0=\mathrm{id}$, this relation is satisfied for $t=0$ and therefore, it is enough to assume that the derivative of the left-hand side vanishes. Show that this is equivalent to 
\[
\omega_0(\dd X_{H}\cdot\dd\Phi^t\cdot u,\dd\Phi^t\cdot v)=\omega_0(\dd X_{H}\cdot\dd\Phi^t\cdot v,\dd\Phi^t\cdot u),
\]
where we have used that $\tfrac{\dd}{\dd t}\dd\Phi^t=\dd X_H\cdot\dd\Phi^t$ (why is this identity true? Hint: interchange the order of the time derivative $\tfrac{\dd}{\dd t}$ and the spatial derivative $\dd$). Conclude by observing that
\[
\omega_0(\dd X_H\,\cdot,\cdot )=\dd\big(\omega_0(X_H,\cdot)\big)=-\dd^2 H,
\]
where $\dd^2 H$ denotes the Hessian of $H$ which is a symmetric bilinear form.
\end{ex}
\subsection{Symplectic geometry in cotangent bundles}
So far we have introduced the symplectic geometry of the Euclidean space. However, this is just an example of a more general construction.
\begin{dfn}
	Let $N$ be a manifold. A symplectic form $\omega$ on $N$ is a closed, non-degenerate two-form. Closed means that its exterior differential vanishes:  $\dd\omega=0$. Non-degenerate means that for all $z\in N$ the skew-symmetric bilinear form $\omega_z:T_zN\times T_zN\to\R$ is non-degenerate. The pair $(N,\omega)$ is called symplectic manifold.	
\end{dfn}
\begin{ex}
Show that the non-degeneracy of $\omega$ implies that $N$ has even dimension $n=2m$.	
\end{ex}
By integration, the symplectic form assigns a sort of signed area to two-dimensional objects in $N$. More generally, for every $k=1,\ldots,m$, we can use $\omega^k$ to assign a sort of volume to $2k$-dimensional objects. Then, non-degeneracy of $\omega$ is equivalent to the fact that $\omega^m$ is nowhere vanishing and hence is a genuine volume form on $N$.

The condition $\dd\omega=0$ can be interpreted as a vanishing of a sort of symplectic curvature. In particular, it implies that all symplectic forms locally look the same and indeed, by Darboux theorem there are local coordinates $(q_1,p_1,\ldots,q_m,p_m)$ such that $\omega=\sum_{j=1}^m\dd p_j\wedge\dd q_j$ coincides with the standard symplectic form in euclidean space. 

Symplectic manifolds give us a natural way of defining Hamilton's equations (a more general class of manifolds where this is possible is represented by Poisson manifolds \cite{Vai94}). Indeed, as in the euclidean case, if $H:N\to\R$ is a smooth function, by the non-degeneracy of $\omega$, we can define the vector field $X_H$ by
\[
\omega(X_H,\cdot)=-\dd H.
\]
We denote by $\Phi_H^t$ the flow on $N$ obtained integrating $X_H$. All properties of Hamiltonian flows which we have seen in $\R^{2m}$ still hold on general symplectic manifolds: $H$ is a first integral (show this), and $\Phi_H^t$ preserves $\omega$ (this follows from Cartan's magic formula). Moreover, if $h\in\R$ is a regular value of $H$, then the Hamiltonian flow, up to time reparametrization, on the smooth hypersurface $S:=\{H=h\}$ depends on $H$ only via $S$. Indeed, $X_H|_S$ is a nowhere section of the one-dimensional distribution $\ker(\omega|_{TS})$ on $S$, where $\omega|_{TS}$ is the restriction of $\omega$ to $S$. Therefore, if $K$ is another Hamiltonian having $S$ as a level set, we have $X_H|_S=fX_K|_S$ for some nowhere vanishing function $f:S\to\R$ and therefore the flows of $H$ and $K$ on $S$ have the same geometric trajectories but with a different time parametrization.  
\begin{ex}
Show that $X_H|_S$ annihilates $\omega|_{TS}$. Prove that $\ker(\omega|_{TS})$ is one-dimensional. Hint: the map $\ker(\omega|_{TS})\to \mathrm{Ann}(TS)$ given by $v\mapsto\omega(v,\cdot)$ is an isomorphism.
\end{ex}

We mention two important classes of examples of symplectic manifolds. The first one is given by the complex projective space $P^n_{\C}$ with the Fubini--Study symplectic form and its smooth complex subvarieties (or more generally by Kähler manifolds). These examples are extremely important in algebraic geometry but will not be discussed here. The second one, which will bring us back to Riemannian geometry, is given by cotangent bundles $T^*M$ of closed manifolds $M$. Recall that the cotangent bundle is the vector bundle $\pi:T^*M\to M$ such that $T^*_xM$ is the dual space of the tangent space $T_xM$ at $x\in M$. If $q$ are local coordinates on $M$ and $(q,p)$ are the corresponding local coordinates on $T^*M$, then we define
\[
\omega_{T^*M}:=\sum_{j=1}^m\dd p_j\wedge \dd q_j.
\]
It turns out that this definition does not depend on the choice of coordinates $q$ (check this) and therefore $\omega_{T^*M}$ yields a well-defined symplectic form on $T^*M$.

Starting from a Riemannian metric $g$ on $M$, we can define a canonical Hamiltonian on $T^*M$: It is the kinetic energy $H_g:T^*M\to\R$, $H_g(p):=\tfrac12|p|^2_g$, where $|\cdot|_g$ is the dual norm given by $g$. The key result is now that Hamiltonian trajectories of $H_g$ are exactly the tangent lifts of geodesics for $g$ after applying the duality isomorphism $\flat:TM\to T^*M$, $v\mapsto g(v,\cdot)$. 
\begin{thm}
The trajectory of the Hamiltonian flow of the function $H_g$ are exactly given by $\flat(\dot\gamma):\R\to T^*M$, where $\gamma:\R\to M$ is a geodesic of $g$.
\end{thm}
Thanks to this result we will be able to compare notions of size coming from Riemannian geometry like systoles and volumes with notions of size coming from symplectic geometry. In order to do so, we need first to introduce what we mean by symplectic size starting from $\R^{2m}$, where the first breakthrough was obtained by Gromov.
\subsection{Symplectic capacities}
Let us consider again the standard symplectic form $\omega_0$ on $\R^{2m}$. For $m=1$, $\omega_0$ is just the area form on the plane. For $m>1$, we still have
\[
\omega^m=m!\cdot\dd\vol_{\!\R^{2m}},
\]
where $\dd\vol_{\!\R^{2m}}$ is the euclidean volume. Therefore, we see that symplectomorphisms are volume-preserving:
\[
\mathrm{Symp}(\R^{2m})\subset\mathrm{Vol}(\R^{2m}):=\{F:\R^{2m}\to\R^{2m}\ |\ F^*\dd\mathrm{vol}_{\R^{2m}}=\dd\mathrm{vol}_{\R^{2m}}\}.
\]
\noindent\!\begin{minipage}[t]{.5\textwidth}
This inclusion motivates us to compare symplectic and volume-preserving geometry by looking at their action on subsets of $\R^{2m}$ for $m>1$, very much as we did in our example with the euclidean and the affine geometry of the plane. Let us consider for instance $B^{2m}(r)\subset\R^{2m}$ the ball of radius $r$ in $\R^{2m}$ and write simply $B^{2m}:=B^{2m}(1)$. What can be said of the image of $B^{2m}$ under a symplectomorphism $\varphi$? Gromov suggests to compare $\varphi(B^{2m})$ with an infinite cylinder $Z^{2m}(r):=\{q_1^2+p_1^2\leq r^2\}\subset\R^{2m}$ of radius $r$. No matter how small $r$ is, there is always a volume preserving transformation $F_r$ (Exercise: Find an explicit one) squeezing $B^{2m}$ inside $Z^{2m}(r)$.
\end{minipage}
\begin{minipage}[t]{.5\textwidth}
	\kern-7pt\centering\includegraphics[width=.9\textwidth]{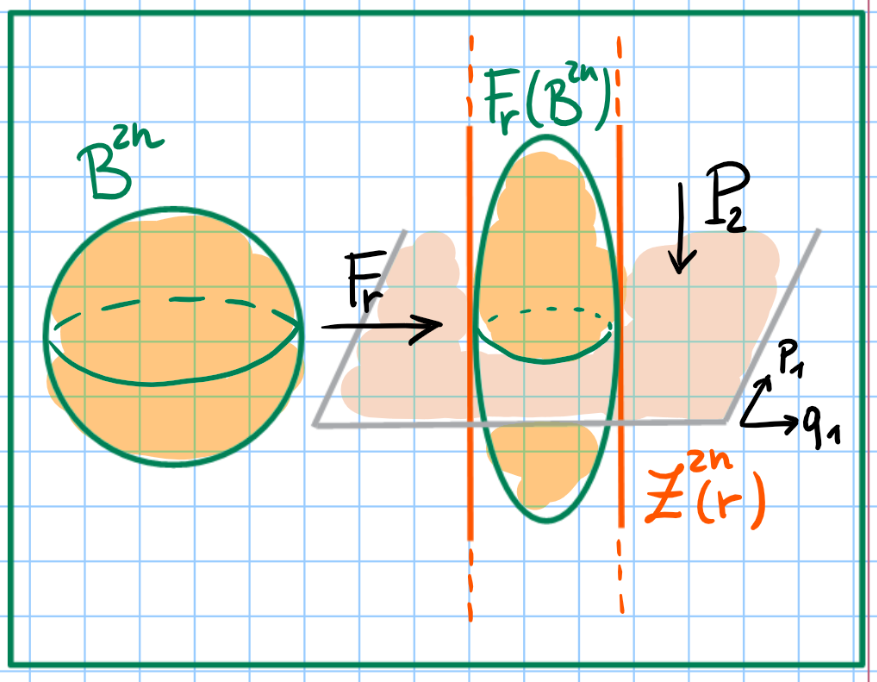}\captionof{figure}{Squeezing the ball}	
\end{minipage}
\medskip

Is the same true for some symplectomorphism? The answer is given by Gromov's non-squeezing theorem.
\begin{thm}[Gromov non-squeezing theorem 1985 \cite{Gro85}]
There is no symplectomorphism $\varphi$ such that $\varphi(B^{2m})\subset Z^{2m}(r)$ for $r<1$.
\end{thm}
An equivalent way of reformulating this amazing result is that the symplectic image of the ball has a two-dimensional shadow of area at least $\pi$:
\begin{equation}\label{e:ns}
\mathrm{vol}_{\R^2}\big(P_{\R^2}(\varphi(B^{2m}))\big)\geq \pi,\qquad\forall\,\varphi\in\mathrm{Symp}(\R^{2m}),
\end{equation}
where $P_{\R^2}:\R^{2m}\to\R^2$ is the projection $(q,p)\mapsto (q_1,p_1)$ on the first two coordinates.
\begin{ex}
Show that \eqref{e:ns} implies Gromov's non-squeezing theorem.
\end{ex}
\begin{ex}
Let $\tilde Z(r):=\{q_1^2+q_2^2\leq r^2\}\subset\R^{2m}$ for $m>1$ and $r>0$. Find $\varphi\in\mathrm{Symp}(\R^{2m})$ such that $\varphi(B^{2m})\subset \tilde Z(r)$.
\end{ex}
We can interpret Gromov's non-squeezing theorem as saying that $B^{2m}$ is symplectically bigger than the infinite cylinder $Z^{2m}(r)$ for $r<1$. This suggests that one should be able to define a symplectic notion of size for subsets of $\R^{2m}$. Which properties should a symplectic size (in the following discussion called capacity) satisfy?
\begin{dfn}
A capacity $c:\mathcal P(\R^{2m})\to[0,\infty]$ assigns to every subset of $\R^{2m}$ a non-negative number according to the rules
\begin{enumerate}[(i)]
\item $c(\varphi(A))=c(A)$, for all $\varphi\in\mathrm{Symp}(\R^{2m})$, $A\subset\R^{2m}$;
\item $c(A_1)\leq c(A_2)$ for all $A_1\subset A_2\subset\R^{2m}$;
\item $c(rA)=r^2c(A)$ for all $r>0$, $A\subset\R^{2m}$;
\item $c(B^{2n})=\pi=c(Z^{2m}(1))$.
\end{enumerate}
If, instead of (iv), $c$ satisfies the weaker assumption
\begin{itemize}
\item[\textit{(iv)'}] $0<c(B^{2m})$ and $c(Z^{2m}(1))<+\infty$,
\end{itemize}
then $c$ is a called a generalized capacity.
\end{dfn}
\begin{rmk}
Observe that (iv) or (iv)' are the crucial properties that enable to detect symplectic phenomena not coming from volume-preserving ones. Indeed, the function $\mathrm{vol}(A)^{\frac{1}{m}}$ satisfies (i), (ii), (iii), is positive on $B^{2m}$ but it is not finite on $Z^{2m}(1)$. 
\end{rmk}
The existence of a capacity is equivalent to Gromov's non-squeezing theorem. Indeed, if $c$ is a capacity and $\varphi(B^{2m})\subset Z^{2m}(r)$, we get
\[
\pi r^2=r^2c(Z^{2m}(1))=c(Z^{2m}(r))\geq c(\varphi(B^{2m}))=c(B^{2m})=\pi,
\]
which implies $r\geq 1$. On the other hand, if Gromov's non-squeezing hold, then the following two quantities are capacities
\begin{align*}
c_{B}(A):&=\sup\{\pi r^2\ |\ \varphi(B^{2m}(r))\subset A,\ \text{for some }\varphi\in\mathrm{Symp}(\R^{2m})\},\\
c_{Z}(A):&=\inf\{\pi r^2\ |\ \varphi(A)\subset Z^{2m}(r),\ \text{for some }\varphi\in\mathrm{Symp}(\R^{2m})\}.
\end{align*}
The first one can be though as a sort of symplectic injectivity radius while the second as a sort of indeterminacy principle of classical mechanics: $c_Z(A)$ gives us a lower bound on how precise the coordinates $q_1$ and $p_1$ of elements of $A$ can be determined, no matter which symplectic coordinates we use.
 
Furthermore, if $c$ is any capacity, there holds
\begin{equation*}
c_{B}\leq c\leq c_Z.
\end{equation*}
\begin{ex}
Show that $c_B\leq c\leq c_Z$ and that $c_B$ and $c_Z$ are capacities using Gromov's non-squeezing theorem.
\end{ex}
Already from the rules in the definition, we can compute the capacity of the region bounded by an ellipsoid. Up to symplectomorphisms \cite{HZ11}, each such region can be written as
\begin{equation}\label{e:ell}
E(a_1,\ldots,a_m):=\Big\{\sum_{i=1}^m\frac{q_i^2+p_i^2}{a_i}\leq 1 \Big\}
\end{equation}
for some positive numbers $a_1\leq \ldots\leq a_m$. 
\begin{ex}\label{e:harm}
Show that the sublevel set of the harmonic oscillator 
\[
\Big\{\frac12|p|^2+\frac12\sum_{i=1}^m\frac{1}{a_i^2}q_i^2\leq \frac12\Big\}
\]
is symplectomorphic to $E(a_1,\ldots,a_m)$.
\end{ex}
By definition, $B^{2m}(\sqrt{a_1})\subset E(a_1,\ldots,a_m)\subset Z(\sqrt{a_1})$ and therefore $E(a_1,\ldots,a_m)=\pi a_1$. We conclude that all capacities coincide on ellipsoids. This is a hint of a far-reaching conjecture about convex bodies like $E(a_1,\ldots,a_m)$. Recall that a convex body is a compact convex subset of $\R^{2m}$ with non-empty interior. 
\begin{con}\label{con:con}
All capacities coincide on convex bodies.
\end{con}
A further evidence in favour of Conjecture \ref{con:con} is the nice observation by Ostrover (see \cite{GHR20}) that $c_B(A)=c_Z(A)$ if $A$ is convex and, after the identification $\R^{2m}=\C^m$ given in Exercise \ref{ex:un}, invariant under multiplication by complex numbers $\mu$ of norm one: $(z_1,\ldots,z_m)\to(\mu z_1,\ldots,\mu z_m)$. 
\begin{ex}
Prove the observation by Ostrover. Hint: consider a ball $B^{2m}(r)$ centered at $0$ of maximal radius $r$ contained in $A$. Then, up to acting with an element of $U(m)$ which commutes with multiplication by complex numbers and preserves the symplectic form, we can suppose that $B^{2m}(r)$ is tangent to $A$ at the point $(r,0,\ldots,0)$. Deduce that $A\subset Z^{2m}(r)$ using that $A$ is convex and invariant under multiplication by complex numbers of norm one. Conclude that $c(A)=\pi r^2$.
\end{ex}
\subsection{A symplectic systolic ratio}
In general, computing a capacity of a convex body is a difficult task. Thus, one could try to estimate the capacity of a convex body with some other notion of size which is easier to compute and which is still invariant under symplectomorphisms. From the discussion around Gromov's non-squeezing, we see that the volume $\vol_{\R^{2m}}(A)$ is a natural candidate.
\begin{que}
Let $c$ be a capacity. Does there exist a positive constant $C>0$ such that for all convex bodies $A\subset\R^{2m}$
\[
\rho_c(A):=\frac{c(A)^m}{m!\vol_{\!\R^{2m}}(A)}\leq C?
\]
If yes, what is the optimal constant $C$? What are the convex bodies maximizing $\rho_c$?
\end{que}
If $A=E(a_1,\ldots,a_m)$, then
\[
\rho_c(A)=\frac{(\pi a_1)^m}{(\pi a_1)\cdot\ldots\cdot(\pi a_n)}\leq 1,
\]
with equality if and only if the $a_i$ are all equal, namely $A$ is a ball. Let us take now an arbitrary convex body $A$. We can consider its John--Loewner ellipsoid $E$ of $A$, namely the ellipsoid of minimal volume containing $A$. Since $A$ is convex and $E$ is of minimal volume, we expect the volumes of $A$ and $E$ not too be much different and indeed by \cite{Vit00}
\begin{equation}\label{e:loewner}
\vol_{\R^{2m}}(E)\leq (4m)^m\vol_{\R^{2m}}(A).
\end{equation}
With this information, one gets
\begin{equation}\label{e:4n}
\rho_c(A)\leq (4m)^m
\end{equation}
which implies that $\rho_c$ is bounded for convex bodies in $\R^{2m}$.
\begin{ex}
Prove inequality \eqref{e:4n} assuming \eqref{e:loewner}.
\end{ex}
This estimate was first proved by Viterbo in \cite{Vit00}. It motivated him to pose the following conjecture.
\begin{con}[Strong Viterbo conjecture \cite{Vit00}]
Let $c$ be a capacity. For all convex bodies $A$ there holds
\[
\rho_c(A)\leq 1\quad\text{and}\quad\Big(\rho_c(A)=1\iff \exists\,r>0,\, \varphi\in\mathrm{Symp}(\R^{2m}),  \ \varphi(A)=B^{2m}(r)\Big).
\]
\end{con}
Further evidence for the conjecture is that $\rho_{c_B}(A)\leq 1$ by the definition of $c_B$. However, knowing $\rho_{c_B}(A)=1$ only implies that an arbitrarily large portion of $A$ is symplectomorphic to a ball and one cannot conclude that the whole $A$ is symplectomorphic to a ball. In particular, Conjecture \ref{con:con} implies $\rho_c(A)\leq 1$ for all convex bodies $A$ (and all capacities $c$). The strong Viterbo conjecture is still open today but the constant $4m$ appearing in \eqref{e:4n} has been improved to a dimension-independent constant.
\begin{thm}[Artstein-Avidan--Milman--Ostrover 2008 \cite{AMO08}]\label{t:vitbounded}
There exists $C>0$ such that for all $m$ and for all capacities $c$ on $\R^{2m}$, there holds
\[
\rho_c(A)\leq C^m,\qquad \text{for all convex bodies }A\subset\R^{2m}.
\]
\end{thm}
\subsection{Periodic orbits of Hamiltonian systems and contact forms}
We mentioned that determining the capacity of a set is a difficult task. For these reasons several capacities have been defined through the years which are easier to compute in specific cases. Beside $c_B$ and $c_Z$, there is an important class of capacities coming from periodic orbits of Hamiltonian systems. In order to introduce it, let us consider an abstract symplectic manifold $(N,\omega)$ and a Hamiltonian function $H:N\to\R$ having compact sublevel sets. Fix a regular energy value $h$ for $H$, so that $S:=\{H=h\}$ is a compact smooth hypersurface. We have seen that the Hamiltonian trajectories on $S$, up to time reparametrization, do not depend on the particular $H$ having $S$ as a level set. If now the form $\omega$ is exact in a neighborhood of $S$, the Maupertuis principle of classical mechanics, periodic orbits of the Hamiltonian flow on $S$ can be described as the critical point of an action functional. To this purpose, let $\lambda$ be a primitive of $\omega$ in a neighborhood of $S$: $\dd\lambda=\omega$. Then, the principle says that a loop $\gamma$ on $S$ with nowhere vanishing velocity is a periodic orbit of the Hamiltonian flow of $H$ up to time reparametrization, if and only if $\gamma$ is a critical point of the action functional
\[
\mathcal A(\gamma):=\int_\gamma\lambda.
\]
As for the length functional in the case of geodesics, we can now define the systole of $S$ by taking the minimum of the action functional among periodic orbits on $S$. The problem is that in general there might be periodic orbits with negative action and the minimum might not be attained and even can be minus infinity. We can remedy to all these problems at once by requiring that the function $\lambda(X_H|_S)\colon S\to\R$ is everywhere positive. Under this condition, there is a way to reparametrize the Hamiltonian flow on $S$ in a way that depends only on the restriction $\alpha:=\lambda|_{TS}$ by considering the flow of the rescaled vector field
\begin{equation}\label{e:rescale}
R_{\alpha}:=\tfrac{1}{\alpha(X_H|_S)}X_H|_S.
\end{equation}
Indeed, the vector field $R_{\alpha}$ can be equivalently defined by the conditions
\[
\alpha(R_{\alpha})=1,\qquad \dd\alpha(R_{\alpha},v)=0,\quad\forall\,v\in TS.
\]
The action $\mathcal A$ of a periodic orbit $\gamma:\R/T\Z\to S_c$ of $R_{\alpha}$ is just the period of $\gamma$:
\[
\mathcal A(\gamma)=\int_\gamma\lambda=\int_0^T\lambda(\dot\gamma)\dd t=\int_0^T\alpha(R_{\alpha})\dd t=\int_0^T1\dd t=T.
\]
On the other hand, if the one-form $\lambda$ is a primitive for $\omega$ on the whole manifold $N$ also the volume of the domain $\{H\leq c\}$ can be written in terms of $\alpha$ by Stokes theorem
\begin{equation}\label{e:vol}
\int_{\{H\leq c\}}\omega^m=\int_{\{H\leq c\}}\dd\Big(\lambda\wedge \omega^{m-1}\Big)=\int_{S}\lambda\wedge \omega^{m-1}=\int_{S}\alpha\wedge (\dd\alpha)^{m-1}.
\end{equation}
We are therefore prompted to give the following abstract definition.
\begin{dfn}
Let $S$ be a smooth compact manifold without boundary of dimension $2m-1$. A one-form $\alpha$ on $S$ is called a contact form if $\ker\dd\alpha$ has dimension one and $\alpha$ is nowhere vanishing on $\ker\dd\alpha$. If $(N,\omega)$ is a symplectic manifold with $S\subset N$ and $\dd\alpha=\omega$ on $S$, then $S$ is called a contact type hypersurface of $N$.
\end{dfn}
\begin{ex}
Show that $\alpha$ is a contact form if and only if $\alpha\wedge\dd\alpha^{m-1}$ is a volume form on $S$, namely it is nowhere vanishing.
\end{ex}
Given a contact form $\alpha$ on an abstract manifold $S$, we can define the so-called Reeb vector field $R_\alpha$ by the relations
\[
\alpha(R_\alpha)=1,\qquad \dd\alpha(R_\alpha,\cdot)=0.
\]
The associated flow is called the Reeb flow. A systole of the contact form is a periodic orbit of the Reeb flow with minimal period which we denote by $T_{\min}(\alpha)$. 
\begin{ex}
Show that $T_{\min}(\alpha)>0$. Hint: show that there is no sequence of periodic Reeb orbits $\gamma_j$ with period $T_j\to 0$. If such a sequence would exist, then $\gamma_j$ converges to some Reeb orbit $\gamma_\infty$ since $S$ is compact. Now observe that since $R_{\alpha}$ has no zeros, there is a neighborhood $U$ of $\gamma_\infty(0)$ and $\tau>0$ such that $\Phi_{R_\alpha}^t(z)\neq z$ for all $z\in U$ and all $t\in(0,\tau)$.
\end{ex}
By definition, $T_{\min}(\alpha)=\infty$ exactly when $R_\alpha$ does not have periodic orbits. However, a central conjecture in symplectic geometry states that such orbits always exist.
\begin{con}[Weinstein 1979 \cite{Wei79}]
Every Reeb flow admits a periodic orbit.
\end{con}
\begin{rmk}
The Weinstein conjecture is still open but has been proven in many important cases. For instance when $S$ has dimension $3$ \cite{Tau07}, when $S$ is contained in $\R^{2m}$ \cite{Vit87} or when $S$ is contained in $T^*M$ and bounds a region containing the zero section in its interior \cite{HV88}. 	
\end{rmk}
The volume of the contact form is given by
\[
\vol_{\!\alpha}(S)=\int_S\alpha\wedge(\dd\alpha)^{m-1}.
\]
Observe that for every $b>0$, $T_{\min}(b\alpha)=bT_{\min}(\alpha)$ and that $\vol_{\!b\alpha}(S)=b^m\vol_{\!\alpha}(S)$. Thus, we define the systolic ratio as
\[
\rho(\alpha)=\frac{T_{\min}(\alpha)^m}{\vol_{\!\alpha}(S)}.
\]
The group of diffeomorphisms $\Psi:S\to S$ naturally acts on a contact form $\alpha$ by pull-back and yields trivially a new contact form $\Psi^*\alpha$ with same minimal period, volume and, hence, systolic ratio
\begin{equation}\label{e:tvrho}
T_{\min}(\Psi^*\alpha)=T_{\min}(\alpha),\quad \vol_{\!\Psi^*\alpha}(S)=\vol_{\!\alpha}(S),\quad \rho(\Psi^*\alpha)=\rho(\alpha).
\end{equation}
How many contact forms, up to diffeomorphism, do we have on a manifold? If $\alpha$ is a contact form, then for every $\psi\colon S\to(0,\infty)$ positive, the one-form $\psi\alpha$ is also a contact form, which is in general not related to $\alpha$ by a diffeomorphism or a rescaling. Notice, indeed, that the dynamics of the Reeb vector fields of $\psi\alpha$ and of $\alpha$ are very different to each other. On the other hand, $\ker\alpha=\ker(\psi\alpha)$ is a co-oriented hyperplane distribution not depending on $\psi$.
\begin{dfn}
A contact structure on $S$ is a co-oriented hyperplane distribution $\xi\subset TS$ such that one (equivalently each) element in $\mathcal C(\xi):=\{\alpha\ |\ \ker\alpha=\xi\}$ is a contact form. The pair $(S,\xi)$ is called a contact manifold.
\end{dfn}
Following the pioneering paper by \'Alvarez-Paiva and Balacheff \cite{AB14}, doing systolic geometry on the contact manifold $(S,\xi)$ means studying the behaviour of the systolic ratio on the set of supporting contact forms
\[
\rho:\mathcal C(\xi)\to(0,\infty].
\]
Contrary to the Riemannian (weak) systolic ratio, the contact one is never bounded.
\begin{thm}[ABHS 2016 \cite{ABHS19} in dimension 3; Saglam 2019 \cite{Sag18} in general]\label{t:unbounded}
For every contact manifold $(S,\xi)$ of dimension at least $3$, the systolic ratio is unbounded on $\mathcal C(\xi)$.
\end{thm}
Still contact systolic geometry has strong connections with the capacity of convex subsets of $\R^{2m}$ and to Riemannian systolic geometry. To uncover these connections, let us look at examples of contact manifolds in $\R^{2m}$ and in $T^*M$.

\subsection{Starshaped hypersurfaces in euclidean space} In $\R^{2m}$ we have a global primitive for the symplectic form. At a point $z\in\R^{2m}$ is defined as
\[
\lambda_0:=\frac{1}{2}\sum_{j=1}^m(p_j\dd q_j-q_j\dd p_j)=\frac{1}{2}\omega_0(z,\cdot).
\]
Thus, for every Hamiltonian $H$, there holds
\begin{equation}\label{e:z}
\lambda_0(X_H)=\tfrac12\omega_0(z,X_H)=\tfrac12\dd H\cdot z.
\end{equation}
\noindent\!\begin{minipage}[t]{.5\textwidth}
If $S=\{H=h\}$, we see that $\lambda_0(X_H|_{S})>0$ if and only if the radial vector field $z$ is transverse to $S$. We call hypersurfaces $S$ with this property starshaped (with respect to the origin). We will also refer to the compact domain $A$ bounded by $S$ also as starshaped. Therefore, starshaped hypersurfaces are of contact type with contact form $\alpha_S:=\lambda_0|_{TS}$.
\end{minipage}
\begin{minipage}[t]{.5\textwidth}
\kern-8pt\centering\includegraphics[width=.7\textwidth]{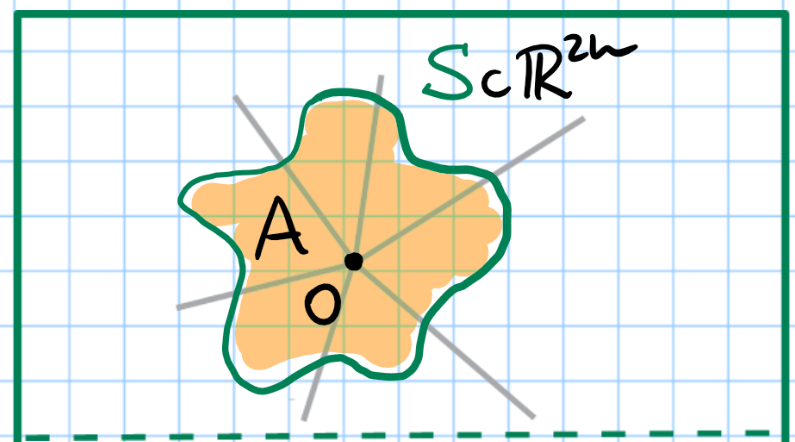}\captionof{figure}{Starshaped domain}
\end{minipage}	\begin{exa}\label{exa:harm}
Suppose that $H:\R^{2n}\to\R$ is two-homogeneous and smooth and positive outside the origin. Then, $S:=H^{-1}(1/2)$ is a starshaped hypersurface and $R_{\alpha_S}=2X_H|_S$. Indeed, using \eqref{e:rescale}, \eqref{e:z} and the Euler theorem for homogeneous functions, we get
\[
\lambda_0(X_H|_S)=\tfrac12\dd H\cdot z=\tfrac122 H|_S=\tfrac12.
\] 
\end{exa}
Starshaped hypersurfaces are in bijection with positive functions $\psi\colon S^{2m-1}\to(0,\infty)$ from the unit sphere. Indeed, if we define
\[
F_\psi:S^{2m-1}\to \R^{2m},\quad F_\psi(z)=\psi(z)z,\qquad S_{\psi}:=F_\psi(S^{2m-1}),
\]
then $S_\psi$ is starshaped and all starshaped hypersurfaces arise in this way. Moreover, there holds
\begin{equation}\label{e:psipsi}
F_\psi^*(\alpha_{S_\psi})=\psi^2\alpha_{S^{2m-1}}.
\end{equation}
Therefore, we also find a bijection between starshaped domains and contact forms in $\mathcal C(\xi_{S^{2m-1}})$, where $\xi_{S^{2m-1}}=\ker\alpha_{S^{2m-1}}$.

Convex bodies containing the origin in the interior are particular examples of starshaped domains. We have the following remarkable property.
\begin{thm}
There exists a capacity $c_{SH}$ such that
\[
c_{SH}(A)\geq T_{\min}(\alpha_{\p A}),\qquad \text{for all starshaped domains }A.
\]
In particular, there exists a periodic orbit on $\p A$. Moreover, 
\[
c_{SH}(A)= T_{\min}(\alpha_{\p A}),\qquad \text{for all convex domains }A.
\]
\end{thm}
Combining this theorem with the volume formula \eqref{e:vol}, we see that for all starshaped domains $A$, there holds
\[
\rho_{c_{SH}}(A)\geq\rho(\alpha_{\partial A})
\]
with equality when $A$ is a convex body. Using Theorem \ref{t:unbounded} and the relation \eqref{e:psipsi}, we see that the systolic ratio with respect to the capacity $c_{SH}$ is unbounded on the set of starshaped domains. On the other hand, by Theorem \ref{t:vitbounded}, this systolic ratio is bounded on the set of convex bodies. Going one (big) step further and proving the Viterbo conjecture for the capacity $c_{SH}$ would already be an amazing result and therefore the statement deserves a name of its own.
\begin{con}[Weak Viterbo conjecture]
For all convex bodies $A$ there holds
\[
\rho_{c_{SH}}(A)\leq 1,\quad\text{and}\quad\Big(\rho_{c_{SH}}(A)=1\iff \exists\,r>0,\, \varphi\in\mathrm{Symp}(\R^{2m}),  \ \varphi(A)=B^{2m}(r)\Big).
\]
\end{con}
We mention here two results implied by the weak Viterbo conjecture.

\subsection{Mahler conjecture}
The first one is the Mahler conjecture in convex geometry, see \cite{Tao07} for an interesting introduction. Before giving the statement, let us recall some definitions about convex bodies $K$ in $\R^m$. First, $K$ is called centrally symmetric if $K$ is invariant by the antipodal map $x\mapsto -x$. Second, the polar of $K$ is the convex body
\[
K^o:=\{ y\in\R^m\ |\ y\cdot x\leq 1,\ \forall\,x\in K\},
\]
where $\cdot$ ist the euclidean product.
Moreover, $K^o$ is centrally symmetric if and only if $K$ is centrally symmetric. Centrally symmetric convex bodies are exactly unit balls of norms on $\R^m$. Taking the polar body then corresponds to considering the dual metric. In particular, the euclidean ball is dual to itself $(B^m)^o=B^m$.

Consider now the product of the volume of $K$ and that of $K^o$ as a function $V$ on the space of centrally symmetric convex bodies in $\R^m$:
\[
K\mapsto V(K):=\vol_{\!\R^m}(K)\cdot \vol_{\!\R^m}(K^o)=\vol_{\!\R^{2m}}(K\times K^o).
\]
\begin{ex}
Prove that $V$ is invariant under affine transformations of $K$.
\end{ex}
Is the function $V$ (also known as the Mahler volume) bounded from above or from below? The first question has a neat answer.
\begin{thm}[Blaschke 1917 \cite{Bla17} for $n\leq3$; Santalo 1949 \cite{San49} in general]
If $K\subset\R^m$ is a convex body, there holds $V(K)\leq V(B^m)$ with equality exactly if $K$ is an affine image of $B^m$, namely an ellipsoid.
\end{thm}
For the second question there is a guess.
\begin{con}[Mahler 1939 \cite{Mah39}]
Let $K$ be a centrally symmetric convex body. Then,
\[
V(K)\geq V(Q^n)=\frac{4^m}{m!},
\]
where $Q$ is a cube.
\end{con}
The conjecture has been proven by Mahler for $m=2$ and very recently by Iriyeh and Shibata for $m=3$. In higher dimension, it is still open and is also expected that the minimizers are exactly the so-called Henner polytopes, which is the smallest class of polytopes containing the interval and closed under products and polarity.

Observe that $K\times K^o$ is a convex body in $\R^{2m}$ although with a non-smooth boundary. Therefore, Mahler's conjecture follows from the weak Viterbo conjecture via the following amazing result.
\begin{thm}[Artstein-Avidan--Karasev--Ostrover 2014 \cite{AKO14}]
For all centrally symmetric convex bodies in $\R^m$, there holds $c_{SH}(K\times K^o)=4$.
\end{thm}

\subsection{Non-squeezing in the intermediate dimensions}
The second statement which is implied by the weak Viterbo conjecture is a non-squeezing theorem in the intermediate dimensions for convex images of the ball. Recalling \eqref{e:ns}, a possible way to ask if there are non-squeezing phenomena for the intermediate volumes $\omega^k_0$ is as follows.
\begin{que}\label{q:nsinter}
Let $1\leq k\leq m$ and $P_{\R^{2m}}:\R^{2m}\to\R^{2k}$ be the projection on the first $2k$ coordinates. Is it true that for every $\varphi\in\mathrm{Symp}(\R^{2m})$, there holds
\begin{equation}\label{e:nsinter}
\vol_{\!\R^{2k}}\big(P_{\R^{2k}}(\varphi(B^{2m}))\big)\geq \vol_{\!\R^{2k}}(B^{2k})=\frac{\pi^k}{k!}?
\end{equation}
\end{que}
The answer is positive for $k=1$ by \eqref{e:ns} and for $k=m$ since symplectomorphisms preserve the volume. However, for every $1<k<m$ Abbondandolo and Matveyev showed in \cite{AM13} that the answer to Question \ref{q:nsinter} is negative in general by constructing for every $\epsilon>0$ a symplectomorphism $\varphi$ with $\vol_{\!\R^{2k}}\big(P_{\R^{2k}}(\varphi(B^{2m}))\big)<\epsilon$. However, the maps $\varphi$ in these counterexamples heavily distort the ball $B^{2m}$, and indeed inequality \eqref{e:nsinter} holds true if $\varphi$ is also assumed to be linear \cite{AM13}, in which case $\varphi(B^{2m})$ is a convex set. Thus, one can speculate if convexity is sufficient to get non-squeezing in intermediate dimensions.
\begin{con}[Convex non-squeezing in the intermediate dimensions]
If $1<k<m$ and $\varphi$ is a symplectomorphism such that $\varphi(B^{2m})$ is convex, then
\[
\vol_{\!\R^{2k}}\big(P_{\R^{2k}}(\varphi(B^{2m}))\big)\geq \frac{\pi^k}{k!}.
\]
\end{con} 
The weak Viterbo conjecture implies the convex non-squeezing in the intermediate dimensions thanks to the additional property that $c_{SH}(P_{\R^{2k}}(A))\geq c_{SH}(A)$ if $A$ is a convex body \cite{AM15} (can you show the implication using this inequality?).

The weak Viterbo conjecture is still open in general but it was recently shown to hold for convex bodies close to the ball.
\begin{thm}[Abbondandolo--Benedetti 2019 \cite{AB19}]\label{t:wV}
Let $A$ be a convex body such that $\p A=S_\psi$ for some $\psi\colon S^{2m-1}$ which is $C^3$ close to the constant function $1$. Then
\[
\rho_{SH}(A)\leq 1
\]
with equality if and only if $A$ is symplectomorphic to a ball.
\end{thm}
Thus, the non-squeezing theorem in the intermediate dimensions holds for symplectomorphisms close to linear ones.
\begin{cor}[Local non-squeezing in the intermediate dimensions]
If $1<k<m$ and $\varphi$ is a symplectomorphism which is $C^3$-close to a linear one, then
\[
\vol_{\!\R^{2k}}\big(P_{\R^{2k}}(\varphi(B^{2m}))\big)\geq \frac{\pi^k}{k!}.
\]	
\end{cor}
\subsection{Fibrewise starshaped hypersurfaces in cotangent bundles}
We come now to examples of convex hypersurfaces in $T^*M$. In $T^*M$, we have the global primitive for $\omega_{T^*M}$ given by
\[
\lambda_{T^*M}=\sum_{j=1}^mp_j\dd q_j.
\]
In particular, we see that 
\begin{equation}\label{e:Y}
	\lambda_{T^*M}=\omega_{T^*M}(Y,\cdot),
\end{equation}
where $Y=\sum_{j=1}^mp_j\partial_{p_j}$ is the fibrewise radial vector field.

\noindent\!\begin{minipage}[t]{.43\textwidth}
\quad As in the euclidean case, for every Hamiltonian $H\colon T^*M\to\R$ there holds
	\begin{equation}\label{e:radial}
		\lambda_{T^*M}(X_H)=\dd H(Y)
	\end{equation}
	so that every $S=\{H=h\}$ which is transverse to the fibrewise radial vector field is of contact type with contact form given by $\alpha_S:=\lambda_{T^*M}|_{TS}$. We call such hypersurfaces fibrewise starshaped (with respect to the zero section). If $S_*$ is a fixed fibrewise starshaped hypersurface, all others hypersurfaces $S$ with this property are obtained from positive smooth functions $\psi\colon S_*\to (0,\infty)$ by defining $S:=S_\psi=F_\psi(S_*)$, where
\end{minipage}
\begin{minipage}[t]{.57\textwidth}
	\kern-8pt
	\centering\includegraphics[width=.9\textwidth]{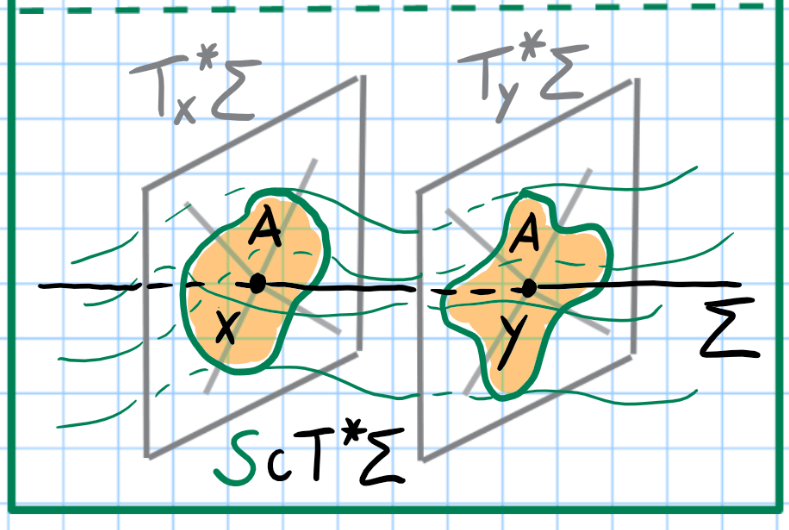}\captionof{figure}{Fibrewise starshaped domain}
\end{minipage}

\[
F_\psi\colon S_*\to T^*M,\qquad F_\psi(\xi)=\psi(\xi)\xi.
\]
There holds
\[
F_\psi^*(\alpha_{S_\psi})=\psi\alpha_{S_*}.
\]
Therefore, contact forms on fibrewise starshaped hypersurfaces are in one-to-one correspondence with the contact forms supporting the contact structure $\xi_{S_*}:=\ker \alpha_{S_*}$ on the fixed hypersurface $S_*$. By Theorem \ref{t:unbounded}, we have the following result.
\begin{cor}\label{c:unboundedcot}
The contact systolic ratio $\rho$ is unbounded on the set
\[
\{\alpha_S\ |\ S\subset T^*M \text{ fibrewise starshaped}\}.
\]
\end{cor}
An important class of fibrewise starshaped hypersurfaces are given by fibrewise convex hypersurfaces, namely the boundaries $S=\p A$ of fibrewise convex domains $A$ containing the zero section in their interior. Such hypersurfaces arise from Finsler norms $F$ on $M$. A Finsler norm $F:TM\to \R$ is a function which restricts to (a possibly asymmetric norm) $F|_{T_xM}\colon T_xM\to\R$ for each $x\in M$. For instance, $F_g:=|\cdot|_g$ is a Finsler norm for every $g\in\mathcal R(M)$. As in the Riemannian case, we can define the length of curves by integrating the norm of the tangent vectors. Curves locally minimizing the length are called $F$-geodesics. Under some assumptions on $F$, we have also a duality $\flat_F:TM\to T^*M$ and Theorem \ref{t:flat} holds also in the Finsler setting, where the Hamiltonian on $T^*M$ is given by $H_F:=\tfrac12 (F^*)^2\colon T^*M\to\R$, where $F^*$ is the dual norm. Thus, we get the fibrewise convex hypersurface
\[
S_F:=\{F^*= 1\}=H_F^{-1}(1/2)
\]
and all fibrewise convex hypersurfaces arise in this way from a suitable Finsler metric. We abbreviate $\alpha_F=\alpha_{S_F}$ and $\alpha_g=\alpha_{F_g}$, $S_g=S_{F_g}$. We know that the Reeb vector field $R_{\alpha_F}$ is a multiple of $X_{H_F}|_{S_F}$. To compute the multiple, we use \eqref{e:radial} and Euler theorem for the fibrewise two-homogeneous function $H_F$:
\[
\lambda_{T^*M}(X_{H_F}|_{S_F})=\dd H_F(Y|_{S_F})=2H_F|_{S_F}=2\cdot\tfrac12=1.
\]
Therefore, the multiple is $1$ and $R_{\alpha_F}=X_{H_F}|_{S_F}$. We obtain the following result.
\begin{thm}\label{t:perper}
A curve $\gamma:\R\to M$ is a geodesic for $F$ with unit speed if and only if $\flat_F(\dot\gamma)$ is a Reeb orbit for $\alpha_F$ on $S_F$. Moreover, $\gamma$ is periodic if and only if $\flat_F(\dot\gamma)$ is periodic and the $F$-length of $\gamma$ is equal to the period of $\flat_F(\dot\gamma)$ (since we took $\gamma$ of unit speed).
\end{thm}
When $F=F_g$, there is also a relationship between the volume of $M$ with respect to $g$ and the contact volume of $\alpha_{g}$, which we do not have time to justify here (see \cite[Exercise 1.32]{Pat99}):
\[
\vol_{\!\alpha_{g}}(S_g)=a_m\vol_{\!g}(M),\qquad a_m:=m!\vol_{\!\R^m}(B^m)
\]
This suggests a way to define the volume of $M$ with respect to a Finsler metric:
\[
\vol_{\!F}(M)=\frac{1}{a_m}\vol_{\!\alpha_F}(S_F).
\]
Consequently, we get the notion of (weak) systolic ratio for $F$:
\[
\sigma(F):=\frac{\mathrm{sys}(F)^m}{\vol_{\!F}(M)},\qquad \rho(F):=\frac{\widetilde{\mathrm{sys}}(F)^m}{\vol_{\!F}(M)},
\]
where $\mathrm{sys}(F)$ is the length of the shortest periodic non-contractible $F$-geodesic and $\widetilde{\mathrm{sys}}(F)$ is the length of the shortest non-constant, periodic $F$-geodesic.
 
In particular, we see that the weak systolic ratio of $F$ and the systolic ratio of $\alpha_F$ differ only by a dimensional constant:
\[
\sigma(F)\geq\rho(F)=a_m\rho(\alpha_F).
\]
In particular, $\rho$ is bounded on the set $\{\alpha_g\ |\ g\in\mathcal R(M)\}$ when the Riemannian (weak) systolic inequality holds on $M$ (see the first talk). This should be compared with Corollary \ref{c:unboundedcot}. 

For each $x\in M$, the set $T_x^*M\cap S_g$ is the boundary of an ellipsoid in $T_x^*M$ since $H_g|_{T_x^*M}$ is a quadratic form. Thus, it seems reasonable that using John--Loewner ellipsoids on every $T_x^*M$, one obtains a systolic inequality also for Finsler metrics provided the Riemannian systolic inequality holds.
\begin{thm}[\'Alvarez-Paiva--Balacheff--Tzanev 2016 \cite{ABT16}]
If the Riemannian systolic inequality holds on $M$, then the systolic ratio $\sigma(F)$ is bounded on the set of Finsler metrics on $M$. In particular, the contact systolic ratio is bounded on
\[
\{\alpha_S\ |\ S\subset T^*M\ \text{fibrewise convex }\}.
\]
Furthermore, the weak systolic ratio $\rho(F)$ is bounded on the set of Finsler metrics on $S^2$. In particular, the contact systolic ratio is bounded on $\{\alpha_S\ |\ S\subset T^*S^2\ \text{fibrewise convex }\}$.
\end{thm}
\begin{rmk}
Observe that the supremum of $\sigma$ on the space of Finsler metrics on $M$ can be strictly bigger than the supremum on the space of Riemannian metrics and even if they coincide there might be maximizing Finsler metrics which are not Riemannian. The first scenario happens on $T^2$ where the supremum is $\pi/2>\sqrt{3}/2$ (This is proved using the solution of the Mahler conjecture in dimension two!). The optimal Finsler metric is constructed as follows: i) Consider the constant Finsler metric over a parallelogram $Q\subset\R^2$ whose unit sphere is given by the very same parallelogram $Q$; ii) obtain a metric on $T^2$ by gluing the opposite sides of the parallelogram $Q$ by translation. \cite[Theorem 12.1]{Sab10}. The second scenario happens on $P^2$ where the supremum is still $\pi/2$ and the optimal metrics are the Zoll Finsler metrics, which build an infinite dimensional space \cite[Theorem 13.1]{Sab10} (compare this with the uniqueness of Zoll Riemannian metrics on $P^2$).
\end{rmk}
\subsection{Zoll contact forms}
We have seen in Theorem \ref{t:unbounded} that $\rho$ is unbounded on $\mathcal C(\xi)$ for every contact manifold $(S,\xi)$. Thus, there is no contact form maximizing the systolic ratio \textit{globally}. It is then natural to ask, if there are contact forms which maximize the systolic ratio \textit{locally}.
\begin{thm}[\'Alvarez-Paiva--Balacheff 2014 \cite{AB14}]\label{t:Zollcont1}
Let $(S,\xi)$ be a contact manifold and $\alpha_0\in \mathcal C(\xi)$ be such that $\infty>\rho(\alpha_0)\geq \rho(\alpha)$ for all $\alpha\in\mathcal C(\xi)$ in a neighborhood of $\alpha_0$. Then, $\alpha_0$ is \textbf{Zoll}, namely all Reeb orbits of $\alpha_0$ are periodic and with the same minimal period.
\end{thm}
\begin{ex}
Prove the theorem filling in the details of the following argument. Suppose that $\alpha_0$ is not Zoll. Then, there is an open set $U\subset S$ such that for every $z\in U$, $\Phi_{R_\alpha}^t(z)\neq z$, for all $t\in(0,T_{\min}(\alpha_0)]$. Consider a non-negative function $f:S\to[0,\infty)$ supported in $U$ and not identically zero. Show that there exists $\epsilon>0$ such that $\alpha_s:=(1+sf)\alpha_0$ satisfies $T_{\min}(\alpha_s)= T_{\min}(\alpha_0)$ (use an argument by contradiction). Finally, use the formula for the volume to show that 
\[
\frac{\dd}{\dd s}\Big|_{s=0}\vol_{\!\alpha_s}(S)=\int_Sf\alpha_0\wedge(\dd\alpha_0)^{m-1}>0.
\]  
Deduce that $\rho(\alpha_s)>\rho(\alpha_0)$ for small negative values of $s$.
\end{ex}
We already encountered many examples of Zoll contact forms. The simplest is $\alpha_{S^{2m-1}}$ the standard contact form on the unit sphere $S^{2m-1}\subset\R^{2m}$. This follows from Remark \ref{r:harm}, Exercise \ref{e:harm} and Example \ref{exa:harm} and we get more precisely \begin{equation}\label{e:hopf}
\Phi_{\alpha_{S^{2m-1}}}(z)=e^{2it}\cdot z,\qquad z\in S^{2m-1}\subset\R^{2m}\cong\C^m.
\end{equation}
Moreover, Theorem \ref{t:perper} shows that $g\in\mathcal R(M)$ is a Zoll Riemannian metric if and only if $\alpha_g$ is a Zoll contact form (the same holds, more generally, for Finsler metrics $F$ on $M$).

Contrary to Zoll Riemannian metrics, the local structure of the space of Zoll contact forms is completely understood. Let $\alpha_0$ be a Zoll contact form and consider a path $s\mapsto \alpha_s$ of contact forms with $s\in(-\epsilon,\epsilon)$. Then, all the $\alpha_s$ are contact forms if and only if there exists a path of diffeomorphisms $s\mapsto\Psi_s:S\to S$ and a path of positive numbers $s\mapsto b_s$ with $\Psi_0=\mathrm{id}$ and $b_0=1$ such that $\alpha_s=b_s\Psi_s^*\alpha_0$. This means that all the deformations of Zoll contact forms are trivially given by rescalings and change of coordinates. On the other hand, we have seen that on $S^2$ there are Zoll metrics obtained deforming the metric of constant curvature $g_*$ which do not have constant curvature. In particular, $s\mapsto \rho(\alpha_s)=\rho(\alpha_0)$ is constant along a path of Zoll contact forms by \eqref{e:tvrho}. The value of $\rho$ at a Zoll contact form has a nice topological interpretation, which gives us a way to classify Zoll contact forms. Indeed, let $\alpha_0$ be Zoll and rescale it so that the minimal period of $\alpha_0$ is $1$. Then the Reeb flow of $\tfrac{1}{T_0}\alpha_0$ induces a free action of the circle $S^1=\R/\Z$ on $S$. We get a quotient map $\mathfrak p:S\to Q$, which is a principal $S^1$-bundle over a $2m-2$-dimensional closed manifold $Q$. The one-form $\alpha_0$ is a connection for this manifold and hence there exists a closed two-form $\kappa_0$ on $Q$, the curvature of $\alpha_0$, with the property that $\dd\alpha_0=\mathfrak p^*\kappa_0$. Since $\alpha_0$ is a contact form $\kappa_0$ is a symplectic form and, being the curvature of the connection, its de Rham cohomology class $e:=[\kappa_0]\in H^2(Q;\Z)$ is the Euler class of $\mathfrak p$. An application of Fubini, then yields
\[
\vol_{\!\alpha_0}(S)=\int_Q\Big(\int_{S^1}\alpha_0\Big)\beta_0^{m-1}=\langle e^{m-1},[Q]\rangle\in\N,
\] 
where $[Q]$ is the fundamental class of $Q$. Thus,
\[
\rho(\alpha_0)=\langle e^{m-1},[Q]\rangle^{-1}.
\]
\begin{rmk}
Let $\alpha_{S^{2m-1}}$ be the contact form induced on $S^{2n-1}\subset\R^{2m}$. By \eqref{e:hopf}, we know that $\alpha_{S^{2m-1}}$ is Zoll with minimal period $\pi$. The quotient by the Reeb flow is the Hopf fibration $\mathfrak p:S^{2m-1}\to P^m_\C$ over the complex projective space and the curvature form is $\tfrac1\pi$-times the Fubini--Study Kähler form.
\end{rmk}
\begin{ex}
Let $g_*$ the round metric on $S^2$ and let $\mathfrak p\colon U^{g_*}S^2\to Q$ be the bundle map obtained quotienting by the Reeb flow of the Zoll contact form $\alpha_{g_*}$. Show that $Q=S^2$ and that $\mathfrak p(\flat \dot\gamma)$ is the center of the great circle parametrized by $\dot\gamma$. 
\end{ex}
An important observation of Boothby and Wang \cite{BW58} is that the above construction can be reversed. If we start with a closed $(2m-2)$-dimensional symplectic manifold $(Q,\kappa_0)$ such that $[\kappa_0]$ is an integral homology class, then there is an $S^1$-principal bundle $\mathfrak p:S\to Q$ and a connection form $\alpha_0$ on $S$ with curvature $\kappa_0$. Since $\kappa_0$ is symplectic, it follows that $\alpha_0$ is a contact form which is Zoll since its Reeb flow induces the $S^1$-action of $\mathfrak p$. If $m=2$, then $Q$ is an orientable surface and we have symplectic forms $\kappa_0$ on $Q$ with any given non-zero class in $H^2(Q;\Z)\cong\Z$. Thus, we get Zoll contact forms on the total space of every non-trivial principal $S^1$-bundle over an orientable closed surface. In particular, there is a Zoll contact form on $U^gQ$ if $Q$ is orientable and different from the two-torus, but there is no Zoll Riemannian metric on $Q$ for genus at least two. Moreover, the systolic ratio of a Zoll contact form on $S$ is independent of the Zoll contact form if $S$ is the total space of a unique principal $S^1$-bundle, up to isomorphism. Indeed, this is what happens when $Q$ is a surface and we recover Theorem \ref{t:wei} by Weinstein, asserting that all Zoll Riemannian metrics on $S^2$ have weak systolic ratio equal to $\pi$.
\begin{ex}
Let $S$ be a closed orientable three-manifold and suppose that there exist principal $S^1$-bundles $\mathfrak p:S\to Q$ and $\mathfrak p':S\to Q'$ for some surfaces $Q,Q'$. Show that $Q$ and $Q'$ are diffeomorphic and that $\mathfrak p$ and $\mathfrak p'$ have the same Euler number, hence they are isomorphic. Hint: use the Gysin sequence of the bundles $\mathfrak p$ and $\mathfrak p'$ to compute the first integral homology of $S$.
\end{ex}
\subsection{Zoll contact forms are local maximizers of the systolic ratio}\label{ss:final}
Theorem \ref{t:Zollcont1} makes us wonder if the converse of its statement is true. To discuss such a converse is the goal of this last subsection.
\begin{thm}[Abbondandolo--Benedetti 2019 \cite{AB19}]\label{t:final}
Let $(S,\xi)$ be a contact manifold and let $\alpha_0\in\mathcal C(\xi)$ be Zoll. If $\alpha\in\mathcal C(\xi)$ is sufficiently $C^3$-close to $\alpha_0$, then
\[
\rho(\alpha)\leq \rho(\alpha_0)\quad\text{and}\quad\Big(\rho(\alpha)=\rho(\alpha_0)\iff \alpha \textrm{ is Zoll}\,\Big).
\]
\end{thm}
\begin{rmk}
The theorem above was already known in several important special cases. \`Alvarez-Paiva and Balacheff showed in \cite{AB14} that $\alpha_0$ is a strict local minimizer along paths $s\mapsto\alpha_s$ which are not tangent to the space of Zoll contact forms at every order. Abbondandolo, Bramham, Hryniewicz and Salom\~ao proved the theorem when $S$ is a three-dimensional manifold such that its quotient by the Reeb flow of the contact form is $S^2$ \cite{ABHS18}. As an immediate consequence, they got Theorem \ref{t:abhs18} on the local systolic maximality of Zoll Riemannian metrics on $S^2$. Finally, Benedetti and Kang proved the theorem more generally when $S$ has dimension three \cite{BK20}.
\end{rmk}
As a Corollary of \ref{t:final}, we immediately get Theorem \ref{t:riemzoll} (and an analogous statement for Finsler metrics), and Theorem \ref{t:wV} about the local weak Viterbo conjecture (after a short argument taking care of the equality case).

The proof of Theorem \ref{t:final} follows a scheme similar to Theorem \ref{t:loewner} about the systolic inequality on $T^2$ that we sketched in the first talk and summarized in Remark \ref{r:important}. If $\alpha\in\mathcal C(\xi)$ is a contact form close to $\alpha_0$ Zoll, the first step is to find a diffeomorphism $\Psi:S\to S$ such that $\beta:=\Psi^*\alpha$ is in a suitable normal form. The second step is to apply an inequality between power means to a foliation of loops on the manifold. In Theorem \ref{t:loewner}, we used the inequality between the minimum, the arithmetic mean and the quadratic mean. Here we get away just with the inequality between the minimum and the arithmetic mean. If $f:X\to(0,\infty)$ is a continuous function on a compact topological space $X$ and $\mu$ is a non-zero, finite Borel measure on $X$, then
\begin{equation}\label{e:mean}
\min f\leq \tfrac{1}{\mu(X)}\int_Xf\dd\mu\quad\text{and}\quad\Big(\min f=\tfrac{1}{\mu(X)}\int_Xf\dd\mu\iff f \textrm{ is constant}\,\Big).
\end{equation}
Describing the exact normal form for $\beta$ will lead us too far away and we refer to \cite[Theorem 2]{AB19} for more details. Here, we will treat the case in which
\begin{equation}\label{e:beta}
\beta=(\psi\circ\mathfrak p)\alpha_0,\qquad \psi:Q\to(0,\infty),
\end{equation}
namely $\beta$ is a multiple of $\alpha_0$ by a function which is constant along the Reeb flow of $\alpha_0$. In the general case that we do not treat here, $\beta=(\psi\circ\mathfrak p)\alpha+\delta$, where $\delta$ is a reminder with certain good properties.

The crucial observation now is to check which Reeb orbits of $\alpha_0$ are still Reeb orbits of $\beta$ possibly with a different time parametrization. To this purpose, we need to have
\begin{equation}\label{e:dbeta0}
\dd\beta(R_{\alpha_0},\cdot)=0.
\end{equation}
Thus, let us take the exterior differential in \eqref{e:beta}:
\begin{equation}\label{e:dbeta}
\dd\beta=(\mathfrak p^*\dd\psi)\wedge\alpha_0+\mathfrak p^*(\psi\kappa_0).
\end{equation}
Recall now that the Reeb orbits of $\alpha_0$ are exactly the fibers of $\mathfrak p$, thus $\dd\mathfrak p\cdot R_{\alpha_0}$ and, by \eqref{e:dbeta}, the condition \eqref{e:dbeta0} is equivalent to
\[
0=(\mathfrak p^*\dd\psi)\wedge\alpha_0(R_{\alpha_0},\cdot)+\mathfrak p^*(\psi\kappa_0)(R_{\alpha_0},\cdot)=-(\mathfrak p^*\dd\psi).
\]
Thus, we see that for all $q\in Q$, the fibre $\mathfrak p^{-1}(q)$ is a periodic orbit of $R_\beta$ if and only if $\dd_q\psi=0$, namely $q$ is a critical point of $\psi$. In this case, the period of the periodic orbit for $R_\beta$ is
\[
\int_{\mathfrak p^{-1}(q)}\beta=\int_{\mathfrak p^{-1}(q)}(\psi\circ\mathfrak p)\alpha_0=\psi(q)\int_{\mathfrak p^{-1}(q)}\alpha_0=\psi(q)\cdot 1=\psi(q).
\]
Choosing $q=q_*$, a minimum for the function $\psi$, we obtain
\begin{equation}\label{e:tmin}
T_{\min}(\beta)\leq \min \psi.
\end{equation}
We estimated the minimal period in terms of $\psi$. Now we compute the volume in terms of $\psi$ using Fubini. Omitting $\mathfrak p$ from the notation, we get
\[
\vol_{\!\beta}(S)=\int_S\psi\alpha_0\wedge(\dd(\psi\alpha_0))^{m-1}=\int_S\psi^m\alpha_0\wedge(\dd\alpha_0)^{m-1}=\int_Q\!\Big(\int_{S^1}\!\!\psi^m\alpha_0\Big)\kappa_0^{m-1}=\int_Q\!\psi^m\kappa_0^{m-1}.
\]
Therefore, recalling that $\int_Q\kappa_0^{m-1}=\langle e^{m-1},[Q]\rangle=\rho(\alpha_0)^{-1}$, we get
\[
T_{\min}(\beta)^m\leq\min\psi^m\leq\frac{1}{\langle e^{m-1},[Q]\rangle}\int_Q\psi^m\kappa_0^{m-1}=\rho(\alpha_0)\vol_{\!\beta}(S),
\]
where the first inequality is \eqref{e:tmin} and the second is \eqref{e:mean}. Thus, $\rho(\alpha)=\rho(\beta)\leq \rho(\alpha_0)$ and, if equality holds, then it must hold also in \eqref{e:mean}, which mean that $\psi$ is constant. This implies that $\beta$ and hence $\alpha$ are Zoll.
\bibliographystyle{amsalpha}
\bibliography{DDTG_bib}
\end{document}